\documentclass[12pt]{amsart}

\usepackage[]{graphicx}
\usepackage{mathrsfs}
\usepackage{amsmath}
\usepackage{amsfonts}
\usepackage{mathtools}
\usepackage{cite}
\usepackage{amsthm}
\usepackage{enumerate}
\usepackage[shortlabels]{enumitem}
\usepackage{mathtools}
\usepackage{amssymb}
\usepackage{appendix}
\usepackage{comment}

\theoremstyle{definition}
\newtheorem{remark}{Remark}
\theoremstyle{definition}
\newtheorem*{definition}{Definition}
\theoremstyle{theorem}
\newtheorem{theorem}{Theorem}
\newtheorem{lemma}[theorem]{Lemma}

\usepackage{xcolor}

\newcommand{\floor}[1]{\left\lfloor #1 \right\rfloor}

\usepackage{algpseudocode,algorithm,algorithmicx}
\newcommand*\Let[2]{\State #1 $\gets$ #2}
\algrenewcommand\algorithmicrequire{\textbf{Input:}}
\algrenewcommand\algorithmicensure{\textbf{Output:}}

\def\torus{\mathbb{T}}
\def\nat{\mathbb{N}}
\def\etalchar{}

\hoffset = - .75truein
\relax 
\bibstyle{alpha}
\bibdata{theorical_ref}
\gdef \@abspage@last{39}

\begin{document}

\title[]{COMPUTING THE INVARIANT CIRCLE AND
ITS STABLE MANIFOLDS
FOR A 2-D MAP BY THE PARAMETERIZATION METHOD: \\
EFFECTIVE ALGORITHMS AND
RIGOROUS PROOFS OF CONVERGENCE}

\author{Yian Yao}
\address{School of Mathematics\\ 
Georgia Institute of Technology \\
686  Cherry St. \\ 
Atlanta GA 30332-160}
\email{yyao93@gatech.edu}
\author{Rafael de la Llave} 
\address{School of Mathematics\\ 
Georgia Institute of Technology \\
686  Cherry St. \\ 
Atlanta GA 30332-160}
\email{rafael.delallave@gatech.edu} 
\thanks{Supported in part by NSF DMS-1800241}

\maketitle

\begin{abstract}

We present and analyze rigorously 
a quadratically convergent algorithm to compute an invariant circle
for 2-dimensional maps 
along with the corresponding
foliation by stable manifolds.
The algorithm is based on solving an
invariance equation
using a quasi-Newton method.

We prove that when the algorithm starts from an initial guess that
satisfies the invariance equation very approximately (depending on
some condition numbers, evaluated on the approximate solution), then
the algorithm converges to a true solution which is close to the
initial guess. The convergence is faster than exponential in smooth
norms.

We also conclude that (in a smooth norm), the distance from 
the exact solution and the approximation is bounded by the initial error. 
This allows validating the numerical approximations (a-posteriori results). 
It also implies the usual persistence formulations since the exact 
solutions of the invariance equation for a model are approximate solutions 
for a similar model. 

The algorithm we present works irrespective of whether the dynamics
on the invariant circle is a rotation or it is phase-locked. The
condition numbers required do not involve any global qualitative
properties of the map.
They are  obtained by evaluating derivatives of the initial
guess, derivatives of the map in a neighborhood of
the guess, performing algebraic operations and
taking suprema.

The proof of the convergence is based on a general Nash-Moser implicit function
theorem specially tailored for this problem. The Nash-Moser 
procedure has unusual properties.  As it turns out, 
the regularity  requirements are not very severe (only 2 derivatives suffice).
We hope that this implicit function 
theorem may be of independent interest and have presented it in a self-contained 
appendix. 

The algorithm in this paper is very practical since it converges quadratically,
and it requires moderate storage and operation count. 
Details of 
the implementation and results of the runs are described in a companion 
paper \cite{YaoL21}. 
\end{abstract}

\textbf{keywords:} 

\section{Introduction} \label{introduction}

In the modern theory  of dynamical systems, the study of the invariant manifold
and their corresponding stable manifolds 
plays a key role.
The dynamics on
these objects organize the dynamics in the whole phase space. 

In this paper, we study attractive (or repulsive) invariant circles in
2-dimensional maps as well as the stable (unstable) manifolds of
points. The collection of such manifolds forms a foliation in a neighborhood of the torus.

We recall that according to the theory of normally hyperbolic
manifolds \cite{Fenichel71, Pesin04},  $W^s_x$, the stable manifolds of a point $x$ in the invariant circle,
are the points whose orbits converge \emph{with a fast enough exponential
rate} to the orbit of $x$. 

\begin{remark}
The paper 
 \cite{W74} defined isochrons as the set of  points with the same asymptotic phase 
on the limit cycle. This is 
not equivalent to the stable manifolds in  the sense 
of normally hyperbolic theory. In the theory of normally hyperbolic manifolds,
the stable leaves are characterized by a fast enough convergence to 
the limit cycle.

When the dynamics in the invariant 
circle contains an attractive 
and a repelling periodic orbit (which are attractive and hyperbolic 
for the full map), 
the points in the plane whose orbit is asymptotic to the stable periodic orbit includes
an open set.  On the other hand, the stable manifolds in the sense of normally 
hyperbolic theory will be one dimensional manifolds. At the periodic orbit, 
the stable manifold in the sense of normally hyperbolic theory is 
the ``strong stable manifold''  in the theory of invariant manifolds 
at fixed points. 

In this the paper, for the sake of having manageable sentences,  we will occasionally 
use  ``isochron''	
to mean \emph{`` leaves of the foliation by the stable manifolds in the sense of normally hyperbolic theory''.}
\end{remark}

The 2-dimensional maps we consider appear in several applications. 
For example, as reductions of higher dimensional systems to 
two-dimensional manifolds after a Neimark-Sacker bifurcation 
\cite{Sacker64, RuelleT71,MarsdenM76}. Another case that motivates us
is the periodic perturbation of a 2-D ordinary differental equation  with a limit cycle.
Such examples are very common in practice. For example,
when oscillating circuits with a limit cicle  are subject to AC forcing
\cite{AndronovVK87, Minorsky62} or in Biology when the circadian
rhythms are subject to external forcing \cite{Winfree01}.
Also when neurons are subject to the periodic forcing of
others \cite{Izhikevich07, ErmentroutT10}

The interpretation of periodic forcing of limit cycles is
useful to keep in mind since the  methods we apply are
inspired by those in \cite{HdlL13}, which considered limit cycles and their
manifolds in 2D autonomous ODEs.  As in \cite{HdlL13}, 
our goal will be to find a system of coordinates that 
turns the dynamics in a neighborhood of the limit cycle into a simple one. 
We will take advantage of several identities to obtain a fast quasi-Newton method. 

\begin{remark}
Passing from 2-D differential equations 
to 2-D maps (or 3-D differential equations) is  non-trivial
since new dynamical phenomena appear. The most notorious one 
is that, for 2-D maps, the  
dynamics in the invariant circle could be  phase-locked. 
That is, the dynamics restricted to the invariant circle
could have an attracting periodic orbit and a repelling one. 

Similarly, 
passing from 2-D maps to  3-D maps involves the  new phenomenon of normal resonances,
which is briefly discussed in \cite{YaoL21}).
\end{remark}

From a more technical point of view in the study of 2-D maps, we do not expect
that the invariant circle or the foliation by stable manifolds of
points are analytic but only
finitely differentiable even if the map 
is analytic (in this paper, we will consider
only analytic mappings) See Section 8.2 of \cite{Llave97} and later in 
this paper. 

On the other hand, each of 
the stable manifolds of a point will be shown to be analytic.
This anisotropic regularity of 
the parameterizations of the foliation by stable manifolds 
-- one of the unknowns in the invariance equation  --  has to be taken into account when 
choosing  the spaces for the formulation of the implicit function theorem. It also 
affects the choices 
of discretizations  in the implementations discussed in  the companion paper
 \cite{YaoL21}. Anisotropic regularity is very typical in the theory
of Normally Hyperbolic Invariant Manifolds (NHIM). It 
so happens that the NHIM has a regularity limited  by ratios of 
rates of convergence while the stable manifolds of a point have 
a regularity limited only by the regularity of the map. This 
anisotropic regularity does not happen in the 2-D ODE case. In \cite{HdlL13} 
it is shown that for 2D analytic ODE, 
both the circle and the foliation by stable leaves are analytic. 
The anisotropic regularity is an important novelty going from 2D ODE 
to 2D maps.

The goal of this paper is to provide a framework 
to study these objects (invariant  circles and their stable foliations) in 2-D maps
in a non-perturbative way which also 
leads to reliable abd efficient  numerical algorithms. The numerical 
algorithms we present and justify here converge to 
the true solution faster than exponentially. 
Hence, mathematical results presented 
here also allow us to validate the results of the numerical algorithms.

The proof of the convergence of the algorithm is 
based on an abstract implicit function theorem of Nash-Moser type 
with some differences from other similar theorems, 
but which we hope could be useful for several 
problems in dynamics and related areas.  See Section~\ref{implicitremark} for 
some comparison with other hard implicit function theorems in the literature. 
The algorithm is based on taking advantage of several cancellations 
that allow to get better estimates. It is shown in \cite{YaoL21} 
that the same cancellations that allow to get better estimates, also 
allow to lower the storage requirements of the algorithm and the operations 
needed for a step. The storage requirements
and the operation count per step are  proportional 
to the number of discretization points.

The numerical algorithms described here have been implemented. 
Details on the implementation, some numerical results
and investigation of phenomena that happen at the boundary of 
validity of our results  are described in 
a companion paper \cite{YaoL21}.

\subsection{Description of the Method} 

Following the idea of the parameterization method
\cite{CFdlL03a, CFdlL03b, CFdlL05, H16}, we formulate an
invariance equation (see \eqref{invariance_origin}). This equation 
has two unknowns:
\begin{itemize} 
\item{a)}  embedings of the circle
and its stable manifolds
\item{ b)} the dynamics of the map restricted to 
the invariant objects (the dynamics on the invariant circle and 
the dynamics on the leaves of the foliations). 
\end{itemize}

This invariance equation~\eqref{invariance_origin} 
expresses that the circle is invariant, 
that the stable foliation  is invariant
 (the leaves of the foliation are not invariant 
but they get sent to another  leaf of the foliation by the dynamics). 

We prove that, given an approximate solution of \eqref{invariance_origin}, we
can evaluate some condition numbers on this approximate solution.
If the error in \eqref{invariance_origin}  is smaller than an 
explicit function of the condition numbers, then there is a true solution of the
invariance equation. Furthermore, the true solution is close to the approximate one.
The condition numbers will be obtained by computing several observations of the 
approximate solution. The condition numbers 
do not involve any global assumptions on the map beyond some estimates 
on the derivatives in a neighborhood of the approximate solution. 
Such results are called \emph{a-posteriori} theorems in the numerical literature
\cite{AinsworthO}.

A-posteriori results imply the usual persistence results under
perturbations of dynamical systems. If one can find  a system
with these structures (invariant circle and its stable manifolds),
then, for  a small perturbation of the system, the original invariant
objects provide an approximate invariant object for the perturbed
system.

The a-posteriori results are also of great use in numerical analysis
since they can provide criteria that ensure that the outputs of
numerical computations -- which are approximate solutions of the
invariance equation -- can be trusted if we supplement them with a 
calculation of the 
condition numbers.  Having very explicit condition numbers and
results that allow trusting the calculation is invaluable 
when studying the phenomena that happen near the breakdown of
the invariant objects and elementary tests (reruns, changing discretizations and
the like)
may get confusing. 
Furthermore, if the evaluation of the errors and the condition
numbers are done taking care of all sources of error (truncation, round
off, etc),  one obtains a computer-assisted proof. Besides their
use in numerical analysis, a-posteriori theorems can be used to
validate the results of other non-rigorous techniques such as
asymptotic expansions (these sophisticated expansions are useful 
in the study of degenerate Neimark-Sacker bifurcations \cite{N59, S64}). 

The way that one often proves an a-posteriori theorem is by describing
an algorithm that given an approximate solution produces an even more
approximate one and then showing that, if one starts from an
approximate enough solution, the process converges.

In our case, we will develop a  modification of 
the standard Newton's method to solve the
invariance equation both for the parameterization of 
the  invariant circle, the invariant
foliation and for their dynamics. We will show that,  when
started from an approximate enough solution, this quasi-Newton method 
converges to a true solution.

To obtain the quasi-Newton method, we start with 
standard Newton method for the functional equation, but take into
account that due to the structure of the problems, there are several
useful identities. Using these identities coming from the geometry
(related to the ``group structure'' in \cite{Moser66a})   we
can obtain an algorithm that is much easier (and much faster and
reliable when implemented numerically) than the straightforward
Newton method without affecting the essential feature of the Newton
method, namely that the error after one step is roughly quadratic
with respect to the original error. It is interesting that the 
same identities that are used to obtain convergence of the rigorous 
proof lead also to a more efficient and reliable algorithm. 
We will refer to this iterative method as a \emph{``quasi-Newton''} method. 

To prove the convergence of the quasi-Newton method, we rely on a Nash-Moser
technique, combining the Newton step with a smoothing step. In the
self-contained Appendix~\ref{appendix}, we present an abstract result,
Theorem~\ref{NMIFT}, which we hope could be applicable in similar
problems.

As we will see, the equation~\eqref{invariance_origin}
is underdetermined.  This underdetermination is quite useful since 
it allows to develop more efficient numerical methods.  As it is
well known, the geometric objects (invariant circle and the stable foliation
are locally unique. The underdeterminacy, is only about the parameterization. 
The same geometric  object can be given different parameterizations. Some of 
them will be numerically more efficient.

\subsection{Some Remarks on Comparison of the Nash-Moser 
Theorem   with Other Results}
\label{implicitremark}

For the experts in Nash-Moser theory, we point out that
Theorem~\ref{NMIFT} developed in Appendix~\ref{appendix} has several unusual
properties.  

Of course, this subsection can be omitted in the first reading,
but we hope this could 
serve as motivation for some of the analysis.

\begin{itemize} 
\item 
The linearized equation 
can be solved without loss of regularity for a range of 
regularities, but there is no theory of solutions for more regular 
data.

This is very different from 
the Nash-Moser applications  in small divisor problems or in PDE, in which one 
can solve the linearized equation 
in spaces of functions 
with any regularity
 (including analytic) but the solution incurs a loss of regularity. 

As a consequence, in our problem, we   cannot use usual smoothing techniques 
of approximating by analytic or $C^\infty$ solutions. 
The only smoothing technique we can use is approximations 
by $C^r$ functions (the so-called $C^r$ smoothing).

We found inspiring 
the abstract implicit function theorems from \cite{Schwartz60} and \cite{Zehnder75}.

\item
We will need to  consider
spaces of functions with mixed
regularity. The functions we will consider are 
$C^r$ smooth in one of the variables ($\theta$), but analytic in the other variable $(s)$. 

The function spaces we use  have two indices, 
one to measure the number of derivatives in 
the first variable and another one to measure the size of analyticity domains 
in the second variable.

 These spaces are indeed forced by the nature of the problem.  
It is known that the invariant circles could be only finitely differentiable 
\cite{Llave97} -- the degree of differentiability 
is limited, 
not just by the 
 regularity of the map, but also by the ratios of the eigenvalues at the periodic orbits.
 On the other hand, the leaves of the stable 
foliation are always analytic.  It is known that similar anisotropic regularities
happen in the theory of normally hyperbolic invariant manifolds. 
 We hope that many of the techniques developed 
here could have wider applicability.

\item 
The nonlinear operator involved in the functional equation  
is basically the composition operator -- which has very unusual 
regularity properties in $C^r$ spaces, see \cite{LlaveO99}. 

This operator maps $C^r$  spaces into themselves. However, it is not differentiable 
from $C^r$ to $C^r$ but it is differentiable when the domain and 
the range are given  other topologies. See 
\cite{LlaveO99} for an exhaustive study. 

Hence, computing the remainder of the functional after a 
correction involves losses of derivatives in $C^r$ spaces.   On the other hand, 
when considering Banach spaces of analytic functions, provided that the 
domains and ranges allow the composition, the composition operator is 
differentiable (even analytic). 

Since the composition operator
appears very commonly in the study of invariance equations in dynamical 
systems, maybe some of the techniques developed in this paper 
may have other applications.

\item
As we will see in the detailed calculations, we will only need 
to smooth in the finite differentiable variable $(\theta)$, 
but we do not need to smooth in the analytic variable $(s)$. 

\item  The iteration we use takes advantage of 
some identities obtained by taking derivatives of the invariance equation.
(From the practical point of view, the use of these identities is crucial 
to obtain quadratically convergent algorithms that require small storage 
and small operation count per step of iteration.) 

This entails that the remainder after applying 
the Newton method  contains a term that involves 
the derivative of remainder of the starting approximation times 
the correction. This term is very common in 
many problems of dynamical systems that are solved taking advantage of 
\emph{automatic reducibility}.  Such terms do not appear in many 
other abstract Nash-Moser theorems. Some abstract theorems that incorporate 
the similar terms appear in
\cite{Vano02,CallejaL10,CallejaCL13}. 
\item
The equation considered is underdetermined 
so that the linearized equation will have a kernel. 
\item
The loss of regularity incurred in our result: Theorem~\ref{NMIFT} 
 is much smaller than 
the loss of regularities in  other abstract hard implicit function theorems. 
\end{itemize}

\begin{remark}
Newton or quasi-Newton 
methods to compute invariant objects with the parameterization method 
have been used for a long time in the numerical 
literature \cite{HdlL06b, CanadellH17a,CanadellH17b, H16, Granados}
since they were found empirically to be efficient, and the solutions obtained 
could be validated using the more conventional methods
(either contraction-based methods \cite{BatesLZ08} or 
topological methods \cite{CapinskiZ15}).  

In implementing Newton methods for invariance equation, turns out that 
out of the box Krylov-Arnoldi, etc. methods doew not to work very 
well since the spectrum of equations involving invariance problems
are invariant under rotations \cite{Mather68, Adomaitis07}.

Note also that the Newton or quasi-Newton  methods are much more 
effective than contraction based graph transform methods, especially when the
contracting exponents are close to one.  This is physically the regime of 
small friction, which is receiving great attention since in many 
practical problems, reducing dissipation is a design goal. 

In the case that  the internal dynamics (denoted by $a$ in later sections)
 is fixed to be a rotation, one can solve
 the cohomology equations by Fourier methods so that the computation remains valid
 even for very weak contraction properties (which would require a large 
number of iterates by graph transform methods). 
This case has been studied in the literature several times.
\cite{CallejaCL13,CanadellH17b}. 
It is interesting that, in this case, the method requires 
the use of small divisors. Even if small divisors are 
not required in the linearized invariance equation, but to keep the 
internal dynamics being a rotation.
\end{remark}

\begin{remark} 
Studying simultaneously the equation for the circle
and the foliation is, paradoxically, more efficient than studying first the
circle and then the foliation. 

The reason for this speedup  is that the approximate solutions for the
foliation are very powerful preconditioners for the invariance
equation for the circle. 
\end{remark}

\begin{remark}
Besides using the Nash-Moser method, there are 
other methods 
that also lead to an a-posteriori format by using a contraction in $C^0$ and 
propagated bounds in higher regularity.
\cite{BatesLZ08}. Such contraction methods may  give better regularity results than the Nash-Moser 
methods presented in this paper,
\end{remark}

\begin{remark} 
The models we consider -- limit cycles subject to periodic
perturbations are known to present regimes of parameters where the
phenomena studied here breaks down and some complex behaviors appear: 
\cite{Levi81,WangY02,WangY03}. The study of the boundary between the
regular behavior presented here and the chaotic behavior is a very
interesting mathematical problem \cite{BonattiDV05}.  Some elements of the
boundary of validity of the results have been explored in
\cite{GoldenY88, Rand92a,Rand92b, Rand92c, HaroL06, HaroL07,
  BjerklovS08, CallejaF12,FiguerasH12}. It is clear that there can be
several interesting phenomena at play and that a systematic
exploration of the boundary will yield a very rich variety of
behaviors.  Inspired by this paper, the numerical algorithms implemented in
\cite{YaoL21}
can, in principle, continue the results in the space of parameters to
reach arbitrarily close to parameters where the objects described here
break down 
(The precise definition of breakdown is somewhat subtle, please refer to
\cite{YaoL21} for more detailed discussions).
One can hope that these numerical explorations of the
frontier of hyperbolicity -- which will require substantial effort -- could yield some
new ideas. Having mathematical tools that allow being confident of
numerical results even if they are unexpected, will be important to
discover new phenomena.
\end{remark} 

\begin{remark} 
In this paper, we will specialize in the case of maps in two dimensions, 
but many of the techniques that we develop -- including the abstract 
implicit function theorem -- applies in any number of dimensions.
The adaptation, however, is 
not completely straightforward since new phenomena may appear, related to 
resonances among normal eigenvalues and the dynamics in the stable foliation 
will have to be more complicated. We hope to come back to this problem, 
but anticipate that the dynamics in the stable manifold has to involve 
more parameters. 

We also  note that in the case of higher dimensional manifolds, there are
more complicated foliations defined in a neighborhood. These foliations are, 
in general not unique, but they have been found useful to describe 
the behavior in a neighborhood of a normally hyperbolic manifold
\cite{BatesLZ00}.  We also call attention to the very interesting 
numerical paper \cite{ChungJ15} and its associated numerical 
package {\tt FOLI8PAK} which deals with similar problems. 
 We hope that the present method can be adapted to 
the study of these manifolds or even to some non-resonant foliations. 
The paper \cite{Szalai19} points that these invariant objects may be useful 
in data reduction (see also \cite{LlaveK19}). We hope to come back to these problems. 
\end{remark} 

\subsection{Organization of the Paper}

In this paper, we have chosen to present the motivation before the main
statement, since the motivation leads to 
a practical algorithm. 
Many of the choices in the precise 
formulation are motivated by the need 
to give a precise formulation 
to the calculations. Of course, 
the readers interested only 
on the precisely formulated 
mathematical results  can skip to Section~\ref{result}.

The paper is organized as follows. In Section~\ref{setup}, we formulate
an invariance equation by the parameterization method, which is the
essential object in this paper. 

The algorithm for solving the invariance
equation is discussed and motivated in Section~\ref{algorithm}. 

The rigorous result in the convergence of the algorithm 
(Theorem~\ref{main})
for the
existence of the solution and the convergence for the algorithm is presented in
Section~\ref{result}.

The proof of Theorem~\ref{main} is presented in Section~\ref{proof}, where
we establish estimates on the ingredients of the algorithm.

The final step of the proof of Theorem~\ref{main} is a
modified version of 
a  Nash-Moser implicit function theorem (Theorem~\ref{NMIFT}) 
which we present in 
the self-contained  Appendix~\ref{appendix}. We hope that Theorem \ref{NMIFT}
can be of
independent interest since it could be applicable to
similar problems.

\section{Setup of the Problem} \label{setup}

In this section, we first briefly introduce the general idea
of the parameterization method (Section~\ref{setup_parameterization}). 
More detailed  discussions about this method are in \cite{H16}.

Then, in a manner inspired by \cite{HdlL13}, we 
formulate an invariance equation~\eqref{invariance_origin} for the 
invariant circle and stable foliation near it. 
 (Section~\ref{setup_invariance}) 

It is important to notice that the invariance equation 
\eqref{invariance_origin} is very underdetermined.  Taking 
advantage of this underdetermination, 
in Section~\ref{setup_undetermincy}, we find a version of 
the invariance equation with extra properties. 
In Section~\ref{setup_isochrones}, we discuss the 
stable manifolds.

\subsection{The General Setting for the Parameterization
  Method for Invariant Objects} \label{setup_parameterization}

We start by describing the general idea of the parameterization method
for finding  invariant manifolds. In the later discussion, we will use the
generalized version which allows to find also invariant foliations. 

In a phase space $\mathscr{A}$, $f: \mathscr{A}
\rightarrow \mathscr{A}$ is a diffeomorphism that generates a discrete dynamical
system, the goal is to find an $f$\text{-}invariant submanifold
$\mathscr{K} \subset \mathscr{A}$, i.e. $f(\mathscr{K}) \subseteq \mathscr{K}$. Consider $K: \Theta \rightarrow \mathscr{A}$
to be an injective immersion from some model manifold $\Theta$ that parameterizes
$\mathscr{K}$. We have that $\mathscr{K}$ is $f$\text{-}invariant if
and only if the following invariance equation holds:
\begin{equation} \label{invariance_old}
  f \circ K(\theta) = K \circ a(\theta),
\end{equation}
where the diffeomorphism $a: \Theta \rightarrow \Theta$ is the internal dynamics on $\Theta$, and
$\theta$ is the local coordinate in $\Theta$ (See Figure \ref{parameterization}).
The goal now becomes solving Equation~\eqref{invariance_old} with $K(\theta), a(\theta)$ as the unknowns.

There are several methods to 
solve equation~\eqref{invariance_old} depending on the class of 
dynamical systems used. 

A widely applicable idea (and the one we will be 
concerned with here) is to apply the Newton (or quasi-Newton)
iterative method to find the correction $\Delta_K(\theta)$ and
$\Delta_a(\theta)$ that improves approximate $K(\theta)$ and $a(\theta)$.
By constructing
an adapted frame $P(\theta)$, and representing $\Delta_K(\theta) = P(\theta)
\phi(\theta)$, solving the Newton method for the equation~\eqref{invariance_old} amounts
solving cohomological equations of the form as in equation~\eqref{coho},
which can then be solved under hyperbolicity assumptions. We will present 
algorithms and establish their convergence.

\begin{remark}
  In the case when the rotation number for the internal dynamics is fixed to be
  a given Diophantine number $\omega$, $a(\theta) = \theta + \omega$ is no longer an
  unknown in \eqref{invariance_old}. On the other hand, one 
has to adjust parameters. See  \cite{CanadellH17a, CanadellH17b} for the
theory for invariant circles. 

An alternative theoretical point of view for the adjustment of
parameters is that, if we consider a family with parameters $\lambda$,
our method will obtain a family of circle mappings
$a_\lambda$. Adjusting parameters $\lambda$ as in
\cite{Moser66a,Moser66b}, we obtain that the map $a_\lambda$ is
smoothly conjugate to a Diophantine rotation.
\end{remark}

\begin{figure}
  \centering
  \includegraphics[width=0.4\textwidth]{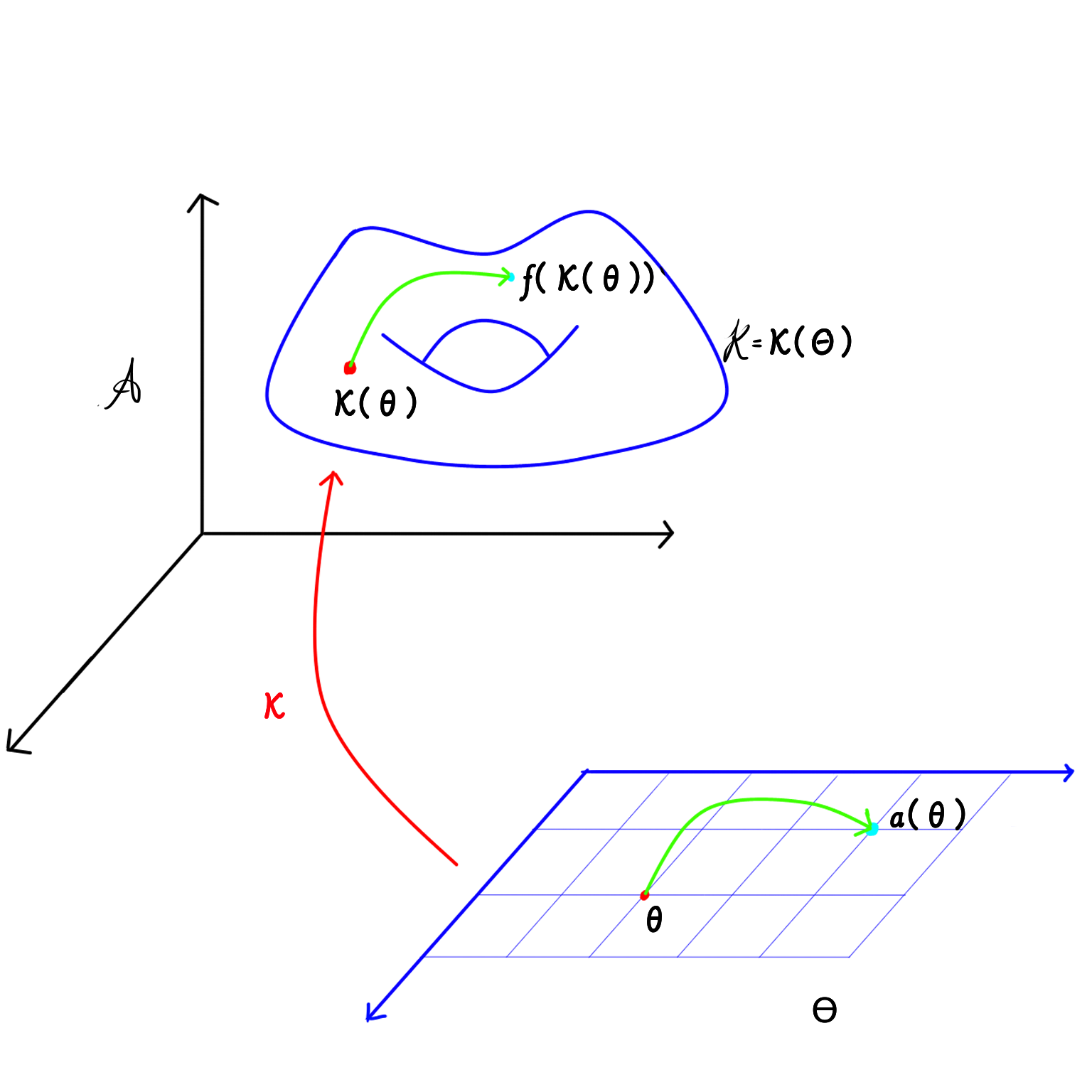}
  \caption{Parameterization of an invariant manifold (Figure taken from \cite{H16})}
  \label{parameterization}
\end{figure}

\subsection{The Invariance Equation of the Invariant Circle and
the Stable Foliation} \label{setup_invariance}

Given a smooth diffeomorphism $f : \mathbb{T} \times
\mathbb{R} \rightarrow \mathbb{T} \times \mathbb{R}$ that generates a discrete
dynamical system in $\mathbb{T} \times \mathbb{R}$, we assume that $f$ admits a
stable invariant circle. Our goal is to find the invariant circle and the
corresponding stable manifolds of points.

More precisely, following a similar approach as in
Section~\ref{setup_parameterization}, we are looking for an injective immersion
$W:
\mathbb{T} \times \mathbb{R} \rightarrow \mathbb{T} \times \mathbb{R}$ such that
it parameterizes the neighborhood of the invariant circle. Thus, we will consider the following invariance equation:
\begin{equation} \label{invariance_origin}
  f \circ W(\theta, s) - W(a(\theta), \lambda(\theta, s)) = 0,
\end{equation}  
where $a: \mathbb{T} \rightarrow \mathbb{T}$ describes the internal dynamics on the invariant circle, and
$\lambda: \mathbb{T} \times \mathbb{R} \rightarrow \mathbb{R}$ describes the
dynamics on the stable manifolds of points.

In the above equation~\eqref{invariance_origin}, $W(\theta, s), a(\theta),
\lambda(\theta, s)$ are the unknowns, and $f(\theta, s)$ is the only known
function.

It is important to emphasize that the unknowns for equation
\eqref{invariance_origin} are functions. Dealing with it in this paper will
require tools from functional analysis. 

Note that when the phase space is 
$\mathbb{T} \times \mathbb{R}$, there are two topologically different embeddings 
of circles. One is when the circle is non-contractible
in the phase space  and the other is 
when the circle is embedded in a contractible way. This can be seen 
as boundary conditions on the embedding $W$.  In the non-contractible 
case, the lift of the embedding satisfies $W(\theta + 1, 0) = W(\theta, 0) + (1,0)$
and in the contractible case, the  lift of the embedding satisfies
$W(\theta + 1, 0) = W(\theta, 0)$. 

It is reasonable to assume that $W(\theta, 0)$ is the parameterization of the
invariant circle, it follows that $\lambda(\theta, 0) = 0$. If one denotes $K(\theta)$ as $W(\theta, 0)$, the invariance equation
\eqref{invariance_origin} reduces to equation~\eqref{invariance_old}.
Moreover, if $\sup_\theta \left| \partial_{s} \lambda(\theta, 0) \right| < 1$,
 we have the
invariant circle is stable. Sharper sufficient conditions for 
stability will be derived later. 

In this paper, we will allow that the internal dynamics $a(\theta)$ is 
 phase-locked (i.e. it has an attractive periodic orbit). 
In such a case, it can happen (indeed, one expects that 
this is the most common case in applications) 
 that the invariant circle is only finitely 
differentiable even if the map $f$ is analytic or even polynomial. 
\subsection{Underdetermination of the Invariance
  Equation} \label{setup_undetermincy}

One nice property of the invariance equation~\eqref{invariance_origin} is that
it is highly underdeterminate, thus admits many solutions.
Hence, depending on the problems, we can impose extra properties 
that improve the computation.   In this section, we will 
review some of the sources of underdetermination that lead to 
improvements in the computation. 

Clearly, the changes of coordinates in the reference manifold
leads the same geometric objects (same circle, same stable leaves) 
but given different parameterizations. 
It can be shown that the only lack of local uniqueness of 
the reference manifold of \eqref{invariance_origin} is these changes of 
variables in the reference manifold. Any two solutions of \eqref{invariance_origin}
close enough are related by a change of variables in the reference manifold and hence 
describe the same geometric object.

{From} the numerical point of view, depending on the properties of 
the system, we can recalibrate the system of coordinates so that the 
computation is better. Clearly, if our goal is to find a solution, 
having several solutions available is a very good feature.

In the following, we review the different sources of 
underdetermination in \eqref{invariance_origin} so that we take advantage of 
them.

\medskip
Given $(W(\theta, s), a(\theta),
\lambda(\theta, s))$ satisfying \eqref{invariance_origin}, we have that
\begin{itemize}
\item Conjugacy on $\theta$:
  For any diffeomorphism $g: \mathbb{T} \rightarrow \mathbb{T}$, we have
  $$\widetilde{W}(\theta, s) = W(g(\theta), s), $$
  $$\widetilde{a}(\theta) = g^{-1} \circ a \circ g(\theta),$$
  $$\widetilde{\lambda}(\theta, s) = \lambda(g(\theta), s),$$
  is also a solution of \eqref{invariance_origin}.
\item Conjugacy on $s$:
  For any $\widehat{\lambda}: \mathbb{T} \times \mathbb{R} \rightarrow
  \mathbb{R}$, if there exists a differentiable function $h: \mathbb{T} \times \mathbb{R} \rightarrow
  \mathbb{R}$ such that
  \begin{equation} \label{undetermination}
    h(a(\theta), \widehat{\lambda}(\theta, s)) = \lambda(\theta, h(\theta, s)),
  \end{equation}
  we have
  $$ \widehat{W}(\theta, s) = W(\theta, h(\theta, s)), $$
  $$\widehat{a}(\theta) = a(\theta),$$
 $$\widehat{\lambda}(\theta, s)$$
 is also another solution of \eqref{invariance_origin}. 
\end{itemize}

According to Lemma~\ref{existence_h2} and its remarks, we can see that
such $h(\theta, s)$ as in \eqref{undetermination} exists in the case that $\widehat{\lambda}(\theta,
s)$ equals to the linear term of $\lambda(\theta, s)$ with respect to
$s$, provided that the norm of $\lambda$ is small enough. We postpone the
detailed discussion and the proof to Section~\ref{proof_unique}. 
As remarked there, Lemma~\ref{existence_h2} 
is a fibered version of Poincar\'e-Sternberg theorem on the linearization of 
contractions. 

Benefiting from the second underdetermination
 property and Lemma~\ref{existence_h2},
instead of considering  \eqref{invariance_origin} we can consider
\begin{equation} \label{invariance}
  f \circ W(\theta, s) - W(a(\theta), \lambda(\theta) s) = 0.
\end{equation}
Our goal now becomes solving for
$W(\theta, s), a(\theta)$ and $\lambda(\theta)$ from equation~\eqref{invariance}.

Note that, if $W,a, \lambda$ is a  solution of \eqref{invariance},
clearly $W$ is an invariant foliation with internal dynamics
given by $a,\lambda$.  One could, however, wonder if there
are other invariant foliations. The content of Lemma~\ref{existence_h2}
is to show that, if there was an invariant foliation, then, one
can obtain a solution of \eqref{invariance} by reparameterizing it.
Hence, finding a solution of \eqref{invariance} is  not only sufficient
for finding invariant foliations but also equivalent.

\begin{remark}
\label{rmk:constantlambda}
  If $a(\theta)$ conjugates to a Diophantine rotation $\theta + \omega$, one can
  show that 
  $\lambda(\theta)$ can be reduced to a constant.

  In fact, given a tuple $(W(\theta, s), \theta + \omega, \lambda(\theta))$
  satisfies Equation~\ref{invariance}, we can show that there exists a constant
  $\bar \lambda$ and $h(\theta, s) = r(\theta) s$ such that Identity
  \eqref{undetermination} holds.

  To prove this, we start with a function $\widehat{\lambda}(\theta)$, and the
  goal is to find $r(\theta)$ to reduce such $\widehat{\lambda}(\theta)$ to a
  constant.  

  From Identity \eqref{undetermination}, we have
  \begin{equation*}
    r(\theta + w) \lambda(\theta) = \widehat{\lambda}(\theta) r(\theta),
  \end{equation*}
  by taking the logrithm, we have
  \begin{equation} \label{log_coho}
    \log \widehat{\lambda}(\theta) = \log \lambda(\theta) + \log r(\theta) - \log r(\theta + \omega),
  \end{equation}

  Standard discussions of cohomological equations in KAM theory \cite{dlL01}
  show that $\widehat{\lambda}$ can be made into a constant if and only if
  \begin{equation}
    \label{coho_for_r}
    \log \lambda(\theta) + \log r(\theta) - \log r(\theta + \omega) = \int_{\mathbb{T}}\log \lambda(\theta) d \theta
  \end{equation}
  holds, in which case $\widehat{\lambda}(\theta) = \exp{\int_{\mathbb{T}}\log \lambda(\theta) d \theta} \triangleq \bar \lambda$.
\end{remark}

\subsection{Stable Manifolds of Points (Isochrons)} \label{setup_isochrones}

Notice that the invariance equation~\eqref{invariance_origin} contains not only
the dynamics of the invariant circle, but also the dynamics in
 a neighborhood
of the invariant circle. In particular, if equation~\eqref{invariance} is
satisfied, and if $\sup_{\theta \in \mathbb{T}}| \lambda(\theta) | < 1$, we have the set
$$I_{\theta} = \{W(\theta, s)\ | \ s \in \mathbb{R}\}.$$
consists of points whose orbits converge exponentially fast (with a high enough rate) 
 to the orbit of
$W(\theta, 0)$ since 
\begin{equation} \label{iterates} 
\begin{split} 
& f^{\circ j}(W(\theta, s)) = W(a^{\circ j}(\theta), \lambda^{[j]}(\theta)s),\\
&f^{\circ j}(W(\theta, 0)) = W(a^{\circ j}(\theta), 0)
\end{split} 
\end{equation} 
where 
\begin{equation}\label{lambdaiterated} 
\lambda^{[j]}(\theta) = \lambda(\theta)\lambda(a(\theta))\lambda(a^{\circ 2}(\theta)) \cdots
  \lambda(a^{\circ (j - 1)}(\theta))
\end{equation} 
 and $a^{j}(\theta)$ denotes $a(\theta)$ composing
  with itself $j$ times.
Note that 
\begin{equation} \label{cocyle} 
\lambda^{[j+k]}(\theta) = \lambda^{[j]}(a^{\circ k}(\theta) ) \lambda^{[k]}
\end{equation} 
Hence $\sup_\theta |\lambda^{[j+k]}(\theta) | \le 
\sup_\theta |\lambda^{[j]}(\theta)| 
\cdot 
\sup_\theta |\lambda^{[k]}(\theta)| 
$.

More specifically, when $\sup_\theta|\lambda(\theta)| < 1$, 
for all $\theta$, we have $f^{k}(I_{\theta}) \rightarrow
  a^{k}(\theta)$ exponentially fast as $n \rightarrow \infty$. 

Note that the isochrons are not invariant sets. Nevertheless, they behave 
well under the map. We have
\[
f(I_\theta) \subset I_{a(\theta)} 
\]
so that the foliation given by all the isochrons is invariant in the sense 
that if two points are in the same leaf, applying the map to both of them, 
we obtain another pair of points in the same leaf (different from the original one). 

\begin{remark}
\label{rmk:dynamicalaverage}
  Given $\lambda(\theta)$, we will refer to the quantity
  \begin{equation} \label{dynamicalaverage} 
    \lambda^* := \lim_{n \rightarrow
      \infty}\Big(\|\lambda^{[n]}\|_{C^0}\Big)^{\frac{1}{n}}
  \end{equation}
  as the \emph{dynamical average}.

  Since $\|\lambda^{[n + m]}\|_{C^0} \leq \|\lambda^{[n]}\|_{C^0}\|\lambda^{[n +
    m]}\|_{C^0}$, the limit in \eqref{dynamicalaverage} always exists.
\end{remark}

The implicit function theorem shows that the set of stable manifolds
 forms a  foliation 
in a neighborhood of the circle and we can use the equation 
\eqref{invariance_origin} to show that the set of isochrons is indeed 
a foliation globally. Note that applying the 
implicit function theorem requires that the circle is $C^1$.
When the circle is less regular, the implicit function theorem 
can only conclude that the leaves form a pre-foliation.
 The conclusion that the isochrons form 
a foliation is also  obtained using more dynamical arguments in 
\cite{Fenichel73,Fenichel77}. It suffices to realize that the relation
\[
y \approx \tilde y \Leftrightarrow  d\big(f^n(y), f^n(\tilde y)\big) \le
C_{y,\tilde y} \lambda^n \quad n > 0
\]
is an equivalence relation. In \cite{Fenichel73, Fenichel77}, it is required
that $\lambda < \|Da^{\circ k}\|$ for some $k > 0$ for the persistence of the circle
as a $C^1$ manifold. 

The phenomena that happen when this inequality is violated, 
have been studied in the literature. A discussion can be found in 
\cite{YaoL21}. 

\begin{comment} 
\begin{remark}
Note that all points in the isochron have the same asymptotic behavior
and that the convergence to this asymptotic behavior is reached exponentially
with a rate $\lambda^*$.  In the theory of normal hyperbolicity, it is standard
that the stable manifold of a point $x$ is the set of points whose asymptotic 
behavior gets closer to that of  $x$ faster than a certain exponential 
rate. It is not enough to require that they just get close. 
For example, in our case, if the dynamics in the invariant circle 
has an attractive periodic point, then, the points whose orbits 
are asymptotic to the periodic orbit is an open set. On the other hand, those
that converge fast enough is  the strong stable manifold of the periodic 
orbit.

One very interesting regime is when the two eigenvalues of 
a periodic point are the same. In such a case, one expects that the circle
stops existing as a $C^1$ manifold \cite{GoldenY88, Mane}, but 
it may persist as a topological circle or as a continuum 
\cite{JarnikK69, CapinskiK20}.   The phenomena that happen near
the breakdown of normal hyperbolicity is still a challenging area. Some
discussions and numerical explorations can be found in \cite{YaoL21}. 
\end{remark}
\end{comment} 

\section{The Algorithm} \label{algorithm}

In this section, we discuss our algorithm for solving the invariance
equation~\eqref{invariance}. Unfortunately, \eqref{invariance} is hard
to solve using the Newton method. 
Instead of the standard Newton
method, we use a modification obtained by  omitting terms
that are heuristically quadratically small. Omitting these terms makes
the equation much easier to solve but, heuristically, does not change
the quadratic convergence.  These heuristic arguments are rigorously
justified later in Section~\ref{proof_proof}.

In Section~\ref{algorithm_derivation} and \ref{solution_quasi_newton}, we present the details of one step of the
quasi-Newton method.

 Given an approximate solution, we look for the
corrections that  so that the new error is quadratic in the original error. 
The method takes advantage of several identities. 

As it turns out, the main ingredient in the method 
is solving \emph{cohomological equations}. The cohomological equations are solved 
in 
Section~\ref{sol_coho_eq}. In Section~\ref{algorithm_algorithm}, we briefly discuss the 
step-by-step algorithm.
In Section~\ref{result} we will state a result on the convergence of  
the algorithm. We will show that the steps can be
repeated infinitely often and indeed converges.

The algorithm formulated in this section has been implemented in \cite{YaoL21}
and run in examples.  We refer to \cite{YaoL21} for details on implementation 
(how to discretize functions, number of variables used)
 As often happens, the algorithm is found to work with even in
regions beyond the requirement of regularity  of the rigorous proof and some new 
phenomena requiring mathematical explanation have been identified. 

\subsection{The quasi-Newton Method} \label{algorithm_derivation}
In this subsection, we formulate one step of the quasi-Newton method to solve
equation~\eqref{invariance}.

 Assume that we have an approximate parameterization of
the neighborhood of the invariant circle $W(\theta, s)$, an approximate
internal dynamics $a(\theta)$ and an approximate dynamics on
the isochrons $\lambda(\theta)s$ such that
\begin{equation} \label{error}
  e(\theta, s) = f \circ W(\theta, s) - W(a(\theta), \lambda(\theta)s),
\end{equation}
where $e(\theta, s)$ is the error.

The goal of one step of the quasi-Newton method is to compute  the corrections
$\Delta_W(\theta, s)$, $\Delta_a(\theta)$ and $\Delta_{\lambda}(\theta)$ such that
\begin{equation} \label{newton0}
  f(W + \Delta_W)(\theta, s) - (W + \Delta_W)((a + \Delta_a)(\theta), (\lambda + \Delta_{\lambda})(\theta)s) = 0
\end{equation}
up to an error which is quadratically smaller than the initial error $e$.

For
the moment, we work heuristically and ignore regularities. All these issues will
be settled later in Lemma~\ref{lemma_condition4}.

Using Taylor expansion and omitting higher order terms, Equation~\eqref{newton0}
becomes
\begin{align} \label{newton}
  0 =& f(W(\theta, s)) + Df(W)(\theta, s) \Delta_W(\theta, s) - W(a(\theta), \lambda(\theta)s) \nonumber \\
  & - DW(a(\theta), \lambda(\theta)s)\begin{pmatrix}
      \Delta_a(\theta) \\
      \Delta_{\lambda}(\theta)s
    \end{pmatrix} 
     - \Delta_W(a(\theta), \lambda(\theta)s) + \text{higher order terms},
\end{align}
where the term $D [\Delta_W(a(\theta), \lambda(\theta)s)] \begin{pmatrix}
  \Delta_a(\theta) \\
  \Delta_{\lambda}(\theta)s
\end{pmatrix}$ is ignored for now because it is ``heuristically'' quadratically
small. We will make a rigorous argument later in Lemma~\ref{lemma_condition4}.

Now we have that Equation~\eqref{newton} has become
\begin{align} \label{newton2}
  Df(W(\theta, s)) \Delta_W(\theta, s)& - DW(a(\theta), \lambda(\theta) s) \begin{pmatrix}
    \Delta_a(\theta) \\
    \Delta_{\lambda}(\theta) s
  \end{pmatrix} \nonumber \\
  - \Delta_W(a(\theta)&, \lambda(\theta) s) = -e(\theta, s).
\end{align}

\begin{remark}
  Notice that one should treat equation~\eqref{newton2} as an equation for 
  $\Delta_W(\theta, s), \Delta_a(\theta)$ and $\Delta_{\lambda}(\theta)$, with
  $f(\theta,
s)$ given by the problem, and $W(\theta, s), a(\theta)$ and $\lambda(\theta)$
given by the initial approximation as well as the RHS $e$. 
\end{remark}
To simplify the above equation~\eqref{newton2}, we will express $\Delta_W(\theta, s)$ in the frame $DW(\theta, s)$ as follows:
\begin{equation} \label{frame}
 \Delta_W(\theta, s) = DW(\theta, s) \Gamma(\theta, s).
\end{equation}

\begin{remark}
  Notice that if $DW(\theta, s)$ is invertible, solving for $\Delta_W(\theta, s)$
  is equivalent to solving for $\Gamma(\theta, s)$. One can see that if the
  initial guess of $W(\theta, s)$ is close enough to the
  true solution and $DW(\theta, s)$ is invertible initially, $DW(\theta, s)$ 
  remains to be invertible for each step of the iteration.
\end{remark}

By taking the derivative of equation~\eqref{error}, we have that
\begin{equation} \label{error2}
  De(\theta, s) = Df(W(\theta, s))DW(\theta, s) - DW(a(\theta), \lambda(\theta) s) \begin{pmatrix} Da(\theta) & 0 \\ D \lambda(\theta) s & \lambda(\theta) 
  \end{pmatrix}.
\end{equation}
Then, by substituting \eqref{frame} and \eqref{error2} in the quasi-Newton
equation~\eqref{newton2}, we obtain
\begin{align} \label{cohom}
  \begin{pmatrix} Da(\theta) & 0 \\ D \lambda(\theta) s & \lambda(\theta) 
  \end{pmatrix} \Gamma(\theta, s) -& \begin{pmatrix}
    \Delta_a(\theta) \\
    \Delta_{\lambda}(\theta) s
  \end{pmatrix} - \Gamma(a(\theta), \lambda(\theta)s) \nonumber \\ 
                             &= - (DW(a(\theta), \lambda(\theta) s)) ^ {-1} e(\theta, s) \nonumber \nonumber \\
                             &\triangleq \widetilde{e}(\theta, s),
\end{align}
where the term $De(\theta, s)\Gamma(\theta, s)$ is also omitted for the same
reason as in equation~\eqref{newton}, and the rigorous justification is again
left to Lemma~\ref{lemma_condition4}.

If we express  equation~\eqref{cohom} in components, we obtain the following two
equations for the unknowns $\Gamma_1(\theta, s)$, $\Gamma_2(\theta, s)$, $\Delta_a(\theta)$ and
$\Delta_{\lambda}(\theta)$.
\begin{equation} \label{eq1}
  Da(\theta) \Gamma_1(\theta, s) - \Delta_a(\theta) - \Gamma_1(a(\theta), \lambda(\theta) s) = \widetilde{e}_1(\theta, s),
\end{equation}
\begin{align} \label{eq2}
  \lambda(\theta) \Gamma_2(\theta, s) - \Delta_{\lambda}(\theta) s - \Gamma_2(a(\theta), \lambda(\theta) s) &= \widetilde{e}_2(\theta, s) - D \lambda(\theta) s \Gamma_1(\theta, s)\\ \nonumber
                                                                                                          &\triangleq M(\theta, s).
\end{align}
where $\Gamma_1(\theta, s)$ and $\Gamma_2(\theta, s)$ are the components of
$\Gamma(\theta, s)$.

\subsection{Solving $\Gamma_{1, 2}, \Delta_ {\lambda}, \Delta_a$ from Equation \eqref{eq1}, \eqref{eq2}} \label{solution_quasi_newton}

In this subsection, we present the details of solving equation~\eqref{eq1} and
\eqref{eq2}.
To study those two equations, we will
discretize any function from $\mathbb{T} \times \mathbb{R}:$ $g(\theta, s)$ as Taylor series with respect to $s$:
$$ g(\theta, s) = \sum_{j = 0}^{\infty}g^{(j)}(\theta) s^j,$$
with the assumption that $g(\theta, s)$ is $C^r$ in $\theta$ and real analytic in $s$,
where $g^{(j)}(\theta) \in C^{r}$ is the coefficient for $s^j, j \geq 0, j \in
\mathbb{N}$. In the context of Section~\ref{result}, $g(\theta, s) \in
\mathcal{X}^{r, \delta}$ for some $\delta > 0$. 

By matching coefficients of $s^j$ on both sides, 
 we can rewrite equation~\eqref{eq1} and \eqref{eq2} 
as a hierarchy of equations provided that $Da(\theta)$ and $\lambda(\theta)$ are
not equal to 0 for any $\theta \in \mathbb{T}$. 

\begin{itemize}
\item For equation~\eqref{eq1}:
  \begin{itemize}
  \item[$\circ$] For the coefficients of $s^0$:
    \begin{equation} \label{eq1_order0}
      Da(\theta) \Gamma_1^{(0)}(\theta) - \Gamma_1^{(0)}(a(\theta)) - \Delta_a(\theta) = \widetilde{e}_1^{(0)}(\theta),
    \end{equation}
  \item[$\circ$] For the coefficients of $s^j$, $j \geq 1, j \in \mathbb{N}$:
    \begin{equation} \label{eq1_higher}
      \Gamma_1^{(j)}(\theta) = \frac{\lambda^j(\theta)}{Da(\theta)} \Gamma_1^{(j)}(a(\theta)) + \frac{\widetilde{e}_1^{(j)}(\theta)}{Da(\theta)}.
    \end{equation}
  \end{itemize}
\item For equation~\eqref{eq2}:
  \begin{itemize}
  \item[$\circ$] For the coefficients of $s^0$:
    \begin{equation*} 
      \lambda(\theta) \Gamma_2^{(0)}(\theta) - \Gamma_2^{(0)}(a(\theta)) = M^{(0)}(\theta),
    \end{equation*}
    which, by composing $a^{-1}(\theta)$, can be rewritten as
    \begin{equation} \label{eq2_order0}
      \Gamma_2^{(0)}(\theta) = \lambda(a^{-1}(\theta)) \Gamma_2^{(0)}(a^{-1}(\theta)) - M^{(0)}(a^{-1}(\theta)),
    \end{equation}
  \item[$\circ$] For the coefficients of $s^1$:
    \begin{equation} \label{eq2_order1}
      \lambda(\theta) \Gamma_2^{(1)}(\theta) - \Gamma_2^{(1)}(a(\theta)) \lambda(\theta) - \Delta_{\lambda}(\theta) = M^{(1)}(\theta),
    \end{equation}
  \item[$\circ$] For the coefficients of $s^j$, $j \geq 2, j \in \mathbb{N}$:
    \begin{equation} \label{eq2_higher}
      \Gamma_2^{(j)}(\theta) = \lambda ^ {j-1}(\theta) \Gamma_2^{(j)}(a(\theta)) + \frac{M^{(j)}(\theta)}{\lambda(\theta)}.
    \end{equation} 
  \end{itemize}
\end{itemize}

The hierarchy of equations above is well known from perturbation expansions. 
Algorithms for efficient computation of the coefficients
can be found in \cite{ZdlL18}.

Again, our goal is to solve the above equations for $\Delta_a(\theta), \Delta_{\lambda}(\theta)$, $\Gamma_1^{(j)}(\theta), \Gamma_2^{(j)}(\theta)$ for $j \geq
0$.

First, notice that Equation~\eqref{eq1_order0} and \eqref{eq2_order1} are
underdetermined equations, hence the solution is not unique. An interesting
question we have not yet pursued 
is how to choose the solution of \eqref{eq1_order0} and 
\eqref{eq2_order1} that 
improves the numerical stability of 
the  algorithm. Intuitively, it seems desirable 
to design the algorithms so that the $a,\lambda$ are 
\emph{``simple''}, but we have not succeeded in 
making this precise when the inner dynamics is phase-locked. (When $a(\theta)$
is conjugate to a Diophantine rotation, one can 
use the underdeterminacy  to set  $a(\theta)$ to be a Diophantine rotation
and $\lambda(\theta)$ to be a constant, see Remark~\ref{rmk:constantlambda}.)

In this paper, we choose the most obvious
solution: For equation~\eqref{eq1_order0}, we let $\Gamma_1^{(0)}(\theta) =
0$ and thus $\Delta_a(\theta) = - \widetilde{e}_1^{(0)}(\theta)$; for equation
\eqref{eq2_order1}, we let $\Gamma_2^{(1)}(\theta) = 0$ and thus
$\Delta_\lambda(\theta) = - M^{(1)}(\theta)$. This choice of solution guarantees
the norm is controled by the error, it is referred as the graph style in
\cite{H16}.

Notice that Equation~\eqref{eq1_higher}, \eqref{eq2_order0} and
\eqref{eq2_higher} have been reorganized so that 
are written as cohomological equation of the form:
\begin{equation} \label{coho}
  \phi(\theta) = l(\theta)\phi(a(\theta)) + \eta(\theta).
\end{equation}
where $\phi(\theta)$ is the unknown and $a(\theta), l(\theta)$ and
$\eta(\theta)$ are given.

\subsection{Solving $\phi$ from the Cohomological Equation \eqref {coho}} \label{sol_coho_eq}

In this subsection, we solve Equation~\eqref{coho} by contraction.

By inductively replacing $\phi(\theta)$ on the right hand side of \eqref{coho}
by the equation itself, we have
\begin{align} \label{coho_steps}
  \phi(\theta) =\  & \eta(\theta) + l(\theta)\eta(a(\theta)) + l(\theta)l(a(\theta))\eta(a^{\circ 2}(\theta)) \nonumber \\ 
                 & + \ldots + l(\theta)l(a(\theta))l(a^{\circ 2}(\theta))\cdots l(a^{\circ (n-1)}(\theta))\eta(a^{\circ n}(\theta)) \nonumber\\
                 & + l(\theta)l(a(\theta))l(a^{\circ 2}(\theta))\cdots l(a^{\circ n}(\theta))\phi(a^{\circ (n+1)}(\theta)) \nonumber\\
  = & \sum_{j = 0}^{n}l^{[j]}(\theta)\eta(a^{\circ j}(\theta)) + l^{[n + 1]}(\theta)\phi(a^{\circ (n+1)}(\theta)),
\end{align}
where as in equation~\eqref{lambdaiterated}, $l^{[j]}(\theta) =
l(\theta)l(a(\theta))l(a^{\circ 2}(\theta)) \cdots
l(a^{\circ (j-1)}(\theta))$, and $l^{[0]}(\theta) = 1$.

Note that, if $\| l^{[j]} \|_{C^0} < 1$, 
and $\phi$ is bounded, the  last term in \eqref{coho_steps} 
tends to zero uniformly. Hence, the only possible 
$C^0$ solution of \eqref{coho}, is 
\begin{equation} \label{coho_solution}
  \phi(\theta) = \sum_{j = 0}^{\infty}l^{[j]}(\theta)\eta(a^{\circ j}(\theta)).
\end{equation}

As proved in Lemma~\ref{lemma_coho} in Section~\ref{proof_coho}, 
 given $r$
such that 
\begin{equation} \label{finite_reg}
\left\|l \right\|_{C^0} \left\|Da \right\|_{C^0}^r < 1,
\end{equation}
we will show  $  \|l^{[j]}(\cdot )\eta(a^{\circ j}(\cdot))\|_{C^r}  \le  C \alpha^j $
for some $C> 0, \alpha <1$ so that 
$\sum_{j = 0}^{n}l^{[j]}(\theta)\eta(a^{\circ j}(\theta))$
converges absolutely
in $C^r$. Hence, \eqref{coho} has a $C^r$ solution. 

The conditions \eqref{finite_reg} can be slightly improved
to $\|l^{[k]}\|_{C^0}  \|D(a^{\circ k}) \|_{C^0}^r < 1$
(or even to 
$\| l^{[k]} D(a^{\circ k}) \|_{C^0} < 1)$.
Nevertheless, there are explicit  examples discussed 
in the remark after Lemma~\ref{lemma_coho}, cohomological equation
\eqref{coho} can only be solved 
for a finite range of $r$.  These  examples are rather persistent 
and they happen in open $C^1$ neighborhoods of $a$. So the phenomenon 
of the quasi-Newton method being defined only on a finite range 
of regularities has to be considered by the Nash-Moser method 
we develop in Appendix~\ref{appendix}.

\begin{remark} \label{faster_coho}
  As we will see in \cite{YaoL21}, the right hand side of equation~\eqref{coho_solution} can be implemented
  very efficiently so that the summation of $M$ terms requires only
  $\log M$ steps.
\end{remark}

\subsection{The Algorithm for One Iteration of the quasi-Newton
  Method} \label{algorithm_algorithm}

By the discussion in Section~\ref{algorithm_derivation}, we now summarize
the steps for one iteration of the quasi-Newton method derived above.
Estimations of the norms will be discussed in Section~\ref{proof}.
Given an approximate solution $W(\theta, s)$, $a(\theta)$ and $\lambda(\theta)$,
where $W(\theta, s)$ is truncated up to the $L$-th order in the power series
expansion (from the analysts' point of view, $L = \infty$), the correction
$\Delta_W(\theta, s)$(with maximal order $L$), $\Delta_a(\theta)$ and
$\Delta_{\lambda}(\theta)$ are calculated by the algorithm as follows:

\begin{algorithm}[htb]
  \caption{One iteration of the algorithm}
  \label{algorithm_n}
  \begin{algorithmic}[1]
    \Require{Initial $W(\theta, s), a(\theta)$ and $\lambda(\theta)$}
    \Ensure{Solution $W(\theta, s), a(\theta)$ and $\lambda(\theta)$ to the
    invariance equation~\eqref{invariance}}
  \Statex
  \Let{$\sum_{j = 0}^Le^{(j)}(\theta)s^j = e(\theta, s)$}{$f \circ
    W(\theta, s) - W(a(\theta), \lambda(\theta)s)$}, 
  \State Compute $DW(\theta, s)$ and $DW \circ (a(\theta), \lambda(\theta)s)$,
  \Let{$\sum_{j = 0}^L\widetilde{e}^{(j)}(\theta)s^j = \widetilde{e}(\theta, s)$}{$
    (DW(a(\theta), \lambda(\theta)s))^{-1} e(\theta, s)$}, 
  \Let{$\Delta_a(\theta)$}{$- \widetilde{e}_1^{(0)}(\theta)$},
  \Let{$\Gamma_1^{(0)}(\theta)$}{0},
  \State Solve $\Gamma_1^{(j)}(\theta)$ from equation~\eqref{eq1_higher} for $1 \leq
  j \leq L$,
  \Let{$\sum_{j = 0}^LM^{(j)}(\theta)s^j = M(\theta, s)$}{$\widetilde{e}_2(\theta,
    s) - D \lambda(\theta) s \Gamma_1(\theta, s)$},
  \Let{$\Delta_{\lambda}(\theta)$}{$- M^{(1)}(\theta)$},
  \Let{$\Gamma_2^{(1)}(\theta)$}{0},
  \State Solve $\Gamma_2^{(0)}(\theta)$ from equation~\eqref{eq2_order0},
  \State Solve $\Gamma_2^{(j)}(\theta)$ from equation~\eqref{eq2_higher} for $2
  \leq j \leq L$, 
  \Let{$\sum_{j = 0}^L\Delta_W^{(j)}(\theta)s^j = \Delta_W(\theta,
    s)$}{$DW(\theta, s) \Gamma(\theta, s)$},
  \Let{$W(\theta, s)$}{$W(\theta, s) + \Delta_W(\theta, s)$},
  \Let{$a(\theta)$}{$a(\theta) + \Delta_a(\theta)$},
  \Let{$\lambda(\theta)$}{$\lambda(\theta) + \Delta_{\lambda}(\theta)$},
  \State Return updated $W(\theta, s), a(\theta)$ and $\lambda(\theta)$.

\end{algorithmic}
\end{algorithm}

\begin{remark}
  Step 6, 10, and 11 are solved based on Equation~\eqref{coho_solution}. As
  stated in Remark~\ref{faster_coho}, in
  \cite{YaoL21}, we present a faster algorithm regarding to this.
\end{remark}

In this paper, we will just present  the analysis of the algorithm 
above and show its convergence under the hypothesis that the 
starting step is close to being a solution.

In 
 \cite{YaoL21} we will discuss the implementation details (discretization, 
programming considerations, and more importantly, diagnostics
of reliability). 

We point out 
that the algorithm is very efficient (it only manipulates
functions). At no time one needs to store (much less invert) a 
matrix with the discretization of the error.  Hence the storage 
requirements will be proportional to space taken by the discretization 
of functions (not the square!) and that the operation count will 
be roughly proportional to the number of variables used to 
discretize a function (there may be logarithmic corrections if 
one uses Fourier methods; see \cite{YaoL21}.)

The algorithm  is also easy to implement in a preliminary -- but 
workable -- form. 
Note that the algorithm has only 16 steps, each of which
can be efficiently implemented in a few lines in a high-level language (or a 
good scientific library).  The most complicated 
step is solving the cohomology equation, but we have iterative formulas for the
solution, along with the quadratic convergence contraction algorithm (more
in \cite{YaoL21}).

Of course, developing a high-quality practical algorithm requires 
developing criteria  that ensure correctness and monitor the 
accuracy.   In that respect, having an 
a-posteriori theorem is an invaluable help. 

The proof of the convergence involves alternating the algorithm with smoothing 
steps. In numerical applications, we have found it convenient to 
include a low pass filter  that smooths the numerical calculations.
This seems to provide enough smoothing.
 See a more detailed discussion in \cite{YaoL21}. 

\section{Scale of Banach Spaces} \label{space}

In this section, we set up the scale of Banach spaces that is needed in
Section~\ref{result} and Section~\ref{proof}. Since for the problem we
are dealing with, functions with domain $\mathbb{T}$ admits only finite
regularity (Lemma~\ref{lemma_coho}), we first recall the $C^r$ space
($r \in \mathbb{N}^+ + (0, 1)$) \cite{LlaveO99} along with some inequalities, and based
on that, we proceed to the
$\mathcal{X}^{r,\delta}$ space for functions in $\mathbb{T} \times \mathbb{R}$,
and finally the $\mathscr{X}^{r, \delta}$ space that will be used in
Section~\ref{result}. The existence of the smoothing operator in $C^r$
guarantees the existence of such smoothing operator in $\mathcal{X}^{r, \delta}$
and $\mathscr{X}^{r, \delta}$ spaces.

\subsection{Setup of the Scale of Spaces} \label{result_setup}
In this section, we describe the spaces that we will use. Roughly speaking, the
spaces are
for functions with domain $(\theta, s) \in \mathbb{T} \times \mathbb{R}$. The functions we 
are interested in will be finitely differentiable in the $\theta$ variable
and analytic in the $s$ variable. The spaces will therefore have 
two indices. One index measuring the -- finite order -- differentiability 
in $\theta$ and another index measuring the size of the analyticity domain in 
$s$. 

The most delicate analysis (smoothing, approximation) will happen in
the finite differentiable direction. In our case, this will be the 
circle. The analysis of finite differentiable spaces
we present is rather standard. As it is well known
in approximation theory, defining a family of regularities
indexed by a real parameter becomes subtle for integer values
of the parameter.  A good reference 
is \cite{Zehnder75, Stein70, LlaveO99}. The properties of 
spaces of functions with mixed regularity used in 
this paper are built on those.

We first recall the $C^r$ spaces.

\subsubsection{Space for Functions in $\mathbb{T}$}
By standard definitions of the H\"older spaces (as in \cite{LlaveO99, dlLW11}), the spaces we will be concerned with
for functions defined on $\mathbb{T}$ are:

\begin{definition}
\label{Crspacese} 
Let $X$ be a Banach space. 

For $r \in \mathbb{N}$, we define: 
\[
C^r(\torus, X)= \{ f: \torus \rightarrow X, \text{$r$ times continuously 
differentiable.} \}
\]
We endow $C^r$ with the supremum norm of all the derivatives of order 
up to $r$, which makes it into a Banach space. 

For $r = n + \alpha \notin \mathbb{N}$ with $n = \floor{r} \in \mathbb{N}, \alpha \in (0,1)$
we define $C^r = C^{n+\alpha}$:
\[
C^r(\torus, X)= \{ f: \torus \rightarrow X, \text{$r$ times continuously differentiable},
D^{\alpha}f \text{ is } \alpha \text{-H\"older.} 
 \}
\]
We endow $C^r(\torus,X)$ with the norm 
\[
\| f \|_{C^{n + \alpha}} =  \max( \|f\|_{C^n}, H_\alpha( D^n f) ),
\]
where for a function $\phi: \torus\rightarrow X$, we set 
\[
H_\alpha(\phi)  =  \sup_{x \ne y} \frac{| \phi(x) - \phi(y) |}{d(x,y)^\alpha}.
\]
\end{definition}

\begin{remark}
  In this paper (excluding Appendix~\ref{appendix}), we always denote $r \geq 0$ for
  the regularity, and we always have $n = \floor{r}$ and $\alpha = r - n$.
\end{remark}

\begin{remark}
  The case $\alpha = 1$
  agrees with the Lipschitz constant and is very natural.  We have excluded it
  to avoid complicating the notation since $C^{r+1}$ would be ambiguous when $r$
  is an integer.  
\end{remark}

\begin{remark}
  $H_\alpha$ is a seminorm and $H_\alpha(\phi) = 0$ if and only if $\phi$ is a
  constant.
\end{remark}

\begin{remark}
  The $C^r$ scale of spaces is very natural and easy to work with since 
the definitions of the norms are very explicit. As it is well known, 
the $C^r$ scale of spaces has anomalies when $r$ is an integer (the 
properties of approximation and smoothing are not as expected, etc). 
So, it is common in analysis to use the other scales of
spaces. (called $\Lambda_r$ in \cite{Stein70} or ${\hat C}^r$ in 
\cite{Moser66a,Zehnder75} ). 

In this paper, we will not use the $\Lambda_r$ spaces (the 
composition operator plays a role in our study and there does 
not seem to be in the literature a systematic study of composition in the
$\Lambda_\alpha$ scales)
but our results 
will include some caveats that the spaces in the hypothesis or
in the conclusions are not integers.  Sometimes, this just amounts to
making some inequalities in the range strict. 
\end{remark}

\begin{remark}  
When $r \in \nat$, the  $C^r$ spaces can be defined taking values in any 
manifold Riemannian (or even Finsler) manifold. When $r>1$ and $r \not\in \nat$, 
the definition, in general, is complicated since to define $H_\alpha$, one
needs to compare the values of derivatives at two different points. This 
requires making explicit some cumbersome choices.   In this paper, 
however, we will only need to deal with $C^r(\torus, \torus)$ or $C^r(\torus,
\mathbb{R})$.
For $\torus$, there is a natural identification of all the tangent
spaces of different points, so that there is no problem in defining $C^r$ spaces taking values 
on the torus. 
\end{remark}

\subsubsection{Space for Functions in $\mathbb{T} \times \mathbb{R}$}
Given $\delta < 1$, we define the space $\mathcal{X}^{r, \delta}$ as follows:
\begin{definition}
  \label{Xr_space}
  For a function $u(\theta, s)$ with domain $\mathbb{T} \times [- \delta, \delta]$, we
  say $u \in \mathcal{X}^{r, \delta}$ if $u(\theta, s) = \sum_{j =
    0}^{\infty}u^{(j)}(\theta)s^j$ with $u^{(j)}(\theta) \in C^r$ and $\sum_{j =
    0}^{\infty}\left\|u^{(j)}\right\|_{C^r}\delta^{j} < \infty$.
  In other words,
  \begin{align*}
    \mathcal{X}^{r, \delta} = \Big\{ u(\theta, s) = &\sum_{j = 0}^{\infty}u^{(j)}(\theta)s^j \mid u^{(j)}(\theta) \in C^r, \text{ and } \sum_{j = 0}^{\infty}\left\|u^{(j)}\right\|_{C^r}\delta^{j} < \infty \Big\}
  \end{align*}
  with norm
  \begin{equation*}
    \left\|u\right\|_{\mathcal{X}^{r, \delta}} = \sum_{j = 0}^{\infty}\left\|u^{(j)}\right\|_{C^r}\delta^{j}.
  \end{equation*}
\end{definition}

\begin{remark}\label{several_spaces} 
It is useful to think of $\mathcal{X}^{r, \delta}$ as a space of 
$C^r$ functions from the circle to a space of analytic functions on the unit
disk.

This corresponds well to the idea of local foliations. We can think of a function 
that to each of the base points associates a segment of
the analytic leaf. 
\end{remark}

\begin{remark}
  Note that the space $\mathcal{X}^{r,\delta}$
  consists of functions that in the variable $s$ have a domain
  of analyticity which is a disk. 

  This is, of course, enough when we are considering local
  foliations, but if we study global foliations, it can
  well happen that the true domain of analyticity of the leaves is
  not a disk.

  From the numerical point of view,  it is
  natural and efficient to represent functions in a disk using power series
  and indeed the definition of the norm in $\mathcal{X}^{r,\delta}$ is
  done to reflect that. On the other hand, one should keep in mind
  that in the global study of foliations, finding solutions of
  \eqref{invariance} in $\mathcal{X}^{r,\delta}$ only gives us segments of
  the leaves. Roughly, we are studying the solution in a circle,
  which extends to the singularity closest to the origin.  If
  this singularity happens away from the real line, the parameterization
  may be analytic for real values outside the circle of convergence.

  Numerically, this corresponds to the step of \emph{``globalization''}.
  Once we have obtained a good representation of the function in a neighborhood
  of the origin using power series, we can use \eqref{invariance} to
  obtain the parameterization in a larger domain. Some interesting examples of
  foliation with global computations appear in \cite{BroerST98}. 
\end{remark} 

For notational simplicity, we denote $\mathcal{X}^{r, \delta}$ as
$\mathcal{X}^{r}$ when the $\delta$ is understood.
We will also not distinguish $\left\| \cdot \right\|_{\mathcal{X}^{r, \delta}}$
and $\left\| \cdot \right\|_{C^r}$ if the space of the analytic function is understood.
Since for $f: \torus \rightarrow \torus$, $f \in C^r$ implies $f \in \mathcal{X}^{r,
  \delta}$ and we have $\left\| f\right\|_{C^r} = \left\| f\right\|_{\mathcal{X}^{r,
    \delta}}$.

\subsection{Basic Properties of $C^r$ and $\mathcal{X}^{r, \delta}$ Spaces} \label{proof_basic}

\subsubsection{Inequalities for Basic Operations}
In this subsection, we present some basic properties and inequalities in
the $C^r$ space. 
\begin{lemma}[Inequalities in $C^r$ Space]
  \label{inequal_Halpha}
  For $\phi, \psi, a\in C^r$, where $r \geq 1$, and $a: \mathbb{T} \rightarrow
  \mathbb{T}$ is a differeomorphism, we have the following inequalities
  \cite{LlaveO99}:
  \begin{enumerate}
  \item $H_\alpha ( \phi \circ a) \leq H_\alpha(\phi)\| Da \|_{C^0}^\alpha$,
  \item $H_\alpha(\phi \cdot \psi)  \le \|\phi\|_{C^0} H_\alpha(\psi) +
    H_\alpha(\phi)  \| \psi\|_{C^0}$,
  \item $\left\| \phi \cdot \psi \right\|_{C^r} \leq 2^{2n+1} \left\| \phi
    \right\|_{C^r} \left\| \psi \right\|_{C^r},$
  \item $\left\| \phi \circ \psi \right\|_{C^r} \leq M_r \left\| \phi \right\|_{C^r} (1 +
    \left\| \psi \right\|_{C^r}^r) \leq 2 M_r \left\| \phi \right\|_{C^r} \left\| \psi
    \right\|_{C^r}^{r},$ where $M_r \geq 1.$
  \end{enumerate}
\end{lemma}

\begin{remark}
  If $a: \torus \rightarrow \torus$ is only $\alpha$-H\"older
  for $\alpha < 1$, the H\"older space is not preserved under composition 
and the best that we can have is $H_{\alpha \beta}( \phi \circ a) \le H_\alpha(\phi) H_\beta(a)^\alpha$.
\end{remark}

Based on Lemma \ref{inequal_Halpha}, we can further derive the following
inequalities. These inequalities will be used in the estimation in
Section~\ref{proof}. We extract them here as an extension to \cite{LlaveO99} and
we hope
they can also be used in other applications.

\begin{lemma}[More Inequalities in $C^r$ Space]
  \label{inequal_more_Cr}
  For $\phi, \psi, a \in C^r$, where $a: \mathbb{T} \rightarrow \mathbb{T}$ is a
  differeomorphism. We assume $k, p, q \in \mathbb{N}^+$. The inequalities are as follows:
  \begin{enumerate}
  \item $H_{\alpha}(D^pa \circ a^{\circ k}) \leq
    H_{\alpha}(D^pa)\|Da\|_{C^0}^{k\alpha}$,
  \item $H_{\alpha}(D(a^{\circ k})) \leq kH_{\alpha}(Da)\|Da\|_{C^0}^{k(\alpha
      + 1) - 1}$,
  \item $\|a^{\circ k}\|_{C^r} \leq k^nn! \|Da\|_{C^0}^{r(k - 1)}
    \|a\|_{C^r}^{r+1},$
  \item $\|\phi(a^{\circ k})\|_{C^r}  \leq
    n!k^{n-1}(n+nk+1)\|\phi\|_{C^r}\|a\|_{C^r}^{r+1}\|Da\|_{C^0}^{kr},$
  \item $\|\psi^{[k]}\|_{C^r} \leq k^{n+1}(n+1)! (\|\psi\|_{C^r} +
    \|a\|_{C^r})^{r+1} \|\psi\|_{C^0}^{\max(0, k-n-1)} \|Da\|_{C^0}^{kr}$, where
    $\|\psi\|_{C^0} < 1$, and as in equation~\ref{lambdaiterated}, $\psi^{[k]} = \psi(a^{\circ (k - 1)})
    \cdots \psi(a) \psi$,
  \item If $\|\psi\|_{C^0} < 1$, $\|\phi(a^{\circ k}) \psi^{[k]}\|_{C^r} \leq
    C_{r} \|\phi\|_{C^r} (\|\psi\|_{C^r}
    + \|a\|_{C^r})^{r+1}\|\psi\|_{C^0}^{-n} k^r
    (\|\psi\|_{C^0}\|Da\|_{C^0}^r)^k,$ 
  \item $\| \phi^k \|_{C^r} \leq k^{2(n-1)} \|\phi\|_{C^r}^{\min(k,
      r)}\|\phi\|_{C^{0}}^{\max(k - n - 1, 0)}.$
  \end{enumerate}
\end{lemma}
\begin{proof}
  By Lemma~\ref{inequal_Halpha}, we have
  \begin{enumerate}
  \item $H_{\alpha}(D^pa \circ a^{\circ k}) \leq H_{\alpha}(D^pa) \|Da^{\circ
      k}\|_{C^0}^{\alpha} \leq H_{\alpha}(D^pa)\|Da\|_{C^0}^{k\alpha}$,
  \item $H_{\alpha}(Da^{\circ k}) \leq k \|Da\|_{C^0}^{k - 1} \max_{0 \leq j
      \leq k} {H_{\alpha}(Da \circ a^{\circ j})} \leq  k 
    H_{\alpha}(Da) \|Da \|_{C^0}^{k(\alpha + 1) - 1}$ ,
  \item Suppose $D^p(a^{\circ k})$ has $T_p$ terms, each term has $F_p$ factors,
    then by
    $$ F_{p + 1} \leq F_p + k - 1, T_{p + 1} \leq T_p F_p \text{ and } F_1 = k,
    T_1 = 1,$$
    we have $F_n \leq nk, T_n \leq k^n(n - 1)!$, for the same $n = \floor{r}$.

    In each term, at most $n(k - 1)$ factors are $Da \circ a^{\circ q}$, at most
    $n$ factors are $D^p(a) \circ a^{\circ q}$, where $0 \leq p \leq n, 0 \leq q
    \leq k$.

    Thus we have
    $$ \|D^na^k\|_{C^0} \leq k^n(n - 1)! \|a\|_{C^n}^n \|Da\|_{C^0}^{n(k-1)}.$$

    We also have
    \begin{align*}
      H_{\alpha}(D^n a^{\circ k}) &\leq k^n(n - 1)! H_{\alpha}(\text{each term in } D^na^k)\\
      &\leq k^n(n - 1)! \Big(n \|Da\|_{C^0}^{n(k-1)} \|a\|_{C^n}^{n - 1} \max_{0 \leq p \leq n, 0 \leq q
        \leq k} H_{\alpha}(D^pa \circ a^{\circ q}) \\
                                  & \ \ + n(k - 1) \|Da\|_{C^0}^{n(k - 1) - 1} \max_{0 \leq q \leq k} H_{\alpha}(Da \circ a^{\circ q}) \|a\|_{C^n}^n \Big)\\
                                  & \leq k^n(n - 1)! \Big(n
                                  \|Da\|_{C^0}^{n(k-1)} \|a\|_{C^n}^{n - 1}
                                  H_{\alpha}(D^n a) \|Da\|_{C^0}^{k\alpha} \\
                                  & \ \ + n(k - 1) \|Da\|_{C^0}^{n(k - 1) - 1} H_{\alpha}(Da) \|Da\|_{C^0}^{k\alpha} \|a\|_{C^n}^n\Big) \\
      & \leq k^{n + 1}n! \|Da\|_{C^0}^{kr - n} \|a\|_{C^r}^{r+1}
    \end{align*}

    Above all, we have
    $$\|a^{\circ k}\|_{C^r} \leq \max\Big(\|a^{\circ k}\|_{C^n}, H_{\alpha}(D^n
      a^{\circ k})\Big) \leq k^nn! \|Da\|_{C^0}^{r(k - 1)} \|a\|_{C^r}^{r+1}.$$

  \item By the same notation and same method as (3), we have $F_n \leq (n + 1)k,
    T_n \leq n!k^{n - 1}$. In each term, at most $n$ factors of $D^pa \circ
    a^{\circ q}$, at most $nk$ factors of $Da \circ a^{\circ q}$ and there is
    a term of $D^p(\phi) \circ a^k$, where $0 \leq p \leq n, 0 \leq q
    \leq k$.

    It follows that 
    $$ \|D^n[\phi(a^{\circ k})]\|_{C^0} \leq n!k^{n -
      1}\|\phi\|_{C^n}\|a\|_{C^n}^n\|Da\|_{C^0}^{nk}, $$
    and
    $$ H_{\alpha}(D^n[\phi(a^{\circ k})]) \leq n! k^{n - 1} (n + nk + 1)
    \|a\|_{C^r}^{r+1} \|Da \|_{C^0}^{kr} \|\phi\|_{C^r}.$$

    Thus, we have
    $$\|\phi(a^{\circ k})\|_{C^r} \leq n!k^{n-1}(n+nk+1)\|\phi\|_{C^r}\|a\|_{C^r}^{r+1}\|Da\|_{C^0}^{kr}.$$

  \item
    By running the same analysis on $\psi^{[k]}$, we have $F_n \leq k(n+1), T_n
    \leq k^nn!$, and for each term, there are at least $\max(k - n, 0)$ factors
    of $\psi$, at most $n$ factors of either $D^pa \circ a^{\circ q}$ or $D^p\psi
    \circ a^{\circ q}$, at most $nk$ factors of $Da \circ a^{\circ q}$,
    where $0 \leq p \leq n, 0 \leq q \leq k$.

    It follows that
    $$ \| D^n \psi^{[k]} \|_{C^0} \leq k^nn! \| \psi\|_{C^0}^{\max(k - n, 0)}
    (\|\psi\|_{C^n} + \|a\|_{C^n})^n \|Da\|_{C^0}^{nk},$$
    and
    $$ H_{\alpha}(D^n\psi^{[k]}) \leq k^{n+1}(n+1)! \|\psi\|_{C^0}^{\max(0,
      k-n-1)} (\|\psi\|_{C^r} + \|a\|_{C^r})^{r+1}\|Da\|_{C^0}^{kr}.$$
    Above all,
    $$ \|\psi^{[k]}\|_{C^r} \leq k^{n+1}(n+1)! (\|\psi\|_{C^r} +
    \|a\|_{C^r})^{r+1} \|\psi\|_{C^0}^{\max(0, k-n-1)} \|Da\|_{C^0}^{kr}.$$
  \item
    Since
    $$ D^n(\phi(a^{\circ k})\psi^{[k]}) = \sum_{q = 0}^n \begin{pmatrix}n \\ q
    \end{pmatrix} D^{n-q}\phi(a^k) D^q \psi^{[k]},
    $$
    and with the previously derived results, we have (with the tedious
    computation omitted), that
    $$ \|\phi(a^{\circ k}) \psi^{[k]}\|_{C^r} \leq C_{r} \|\phi\|_{C^r} (\|\psi\|_{C^r}
    + \|a\|_{C^r})^{r+1}\|\psi\|_{C^0}^{-n} k^r
    (\|\psi\|_{C^0}\|Da\|_{C^0}^r)^k,$$
    where $C_{r}$ is formed by only the power
    series and factorials of $r$.

  \item
    As for $\|\phi^k\|_{C^r}$, we know $D^n (\phi^k)$ has $k^{n-1}$ terms, each
    term has $k$ factors, and each term has at most $\min(k, n)$ factors of $D^p
    \phi$ with the rest of the terms are $\phi$, we have
    $$ \| \phi^k \|_{C^r} \leq k^{2(n-1)} \|\phi\|_{C^r}^{\min(k,
      r)}\|\phi\|_{C^{0}}^{\max(k - n - 1, 0)}.$$
  \end{enumerate}
\end{proof}

Lemma \ref{inequal_Halpha} also implies the following inequality in
$\mathcal{X}^{r, \delta}$ space.
\begin{lemma}[Inequalities in $\mathcal{X}^{r, \delta}$ space]
  \label{lemma_inequalityXr}
  Given $f, g \in \mathcal{X}^{r, \delta}$ we have
  \begin{itemize}
    \item $\left\| f \cdot g \right\|_{\mathcal{X}^{r, \delta}} \leq 2^{2n+1} \left\| f
      \right\|_{\mathcal{X}^{r, \delta}} \left\| g \right\|_{\mathcal{X}^{r, \delta}},$
  \end{itemize}
\end{lemma}

By \cite{LlaveO99}, the $C^r(\mathbb{T}, X)$ space, thus the $\mathcal{X}^{r,
  \delta}$ space we are considering in this paper admits a scale of Banach
Spaces with continuous inclusion. In other word, for $0 \leq r \leq s$, we have
$C^s(U, X) \subset C^r(U, X)$ and $\mathcal{X}^{s, \delta} \subset
\mathcal{X}^{r, \delta}$.

\begin{remark}
  Generally speaking, the scale of spaces $C^r(U, X)$ does not admits continuous
  inclusion for general domain $U$ (counterexample can be found in
  \cite{LlaveO99}). The continuous inclusion is guaranteed when $U$ is a
  compensated open set \cite{LlaveO99}. In our case, the domain of the functions
  is torus, which is a simple compensated open set.
\end{remark}

\subsubsection{Smoothing Operators}

To develop the Nash-Moser smoothing technique, for a scale of Banach spaces
$\mathcal{X}^{r, \delta}$, we need the
existence of a family of smoothing operators defined as follows:
\begin{definition}[Smoothing Operator]
\label{def:abstract_smoothing}
  For a scale of Banach spaces $X_r$, a family of smoothing operators $\{S_t\}_{t \in \mathbb{R}^+}$ satisfies
  \begin{equation} \label{smoothing1}
    \left\| S_t u\right\|_{\mu} \leq t^{\mu - \lambda} C_{\lambda, \mu}
    \left\| u\right\|_{\lambda} \text{ for } u \in X_{\lambda}
  \end{equation}
  and 
  \begin{equation} \label{smoothing2}
    \left\| (S_t - I) u\right\|_{\lambda} \leq t^{-(\mu - \lambda)} C_{\lambda, \mu}
    \left\| u\right\|_{\mu} \text{ for } u \in X_{\mu}
  \end{equation}
  for $\mu \geq \lambda \geq 0$, where $t$ is the strength of smoothing.
\end{definition}

\begin{remark} \label{C_r_smoothing}
  When $X$ is a Banach space, the existence of the $C^{r}$-smoothing operator in $C^r(\mathbb{T}, X)$
  is studied in \cite{Zehnder75}.
\end{remark}

With such smoothing operator in $C^r$ space, we can define the smoothing
operator for a function $u(\theta, s) = \sum_{j = 0}^{\infty}u^{(j)}(\theta)s^j
\in \mathcal{X}^{r, \delta}$ by smoothing each of $u^{(j)}(\theta)$ for $j \geq 0$. More
precisely, we have
\begin{definition}[Smoothing Operator in $\mathcal{X}^{r, \delta}$]
  For $u(\theta, s) = \sum_{j=0}^{\infty} u^{(j)}(\theta) s^j \in
  \mathcal{X}^r$, the smoothing operator $S_t$ is defined
  as follows:
  \begin{equation} \label{smoothing0}
    S_t u(\theta, s) = \sum_{j = 0}^{\infty}\hat{S}_tu^{(j)}(\theta)s^j.
  \end{equation}
  where $\hat{S}_t$ is the smoothing operator in $C^r$ space defined in Remark \ref{C_r_smoothing}.
\end{definition}

In our problem, since \eqref{invariance} has unknowns which are 
triples of functions, $(W,a,\lambda)$, we will see that the smoothing operators
defined so far, lead straightforwardly to smoothing operators in the space of 
triples.  See Section~\ref{sec:triples}. 

Note that the definition of smoothing in $\mathcal{X}^{r, \delta}$ defined above 
is the standard $C^r$ smoothing applied spaces of $C^r$ functions 
taking values in a space of analytic functions as discussed in 
Remark~\ref{several_spaces}.

It is standard  to see that this operator $S_t$ defined in
\eqref{smoothing0} satisfies condition \eqref{smoothing1} and
\eqref{smoothing2}, thus it is indeed a smoothing operator in $\mathcal{X}^{r, \delta}$.

\begin{remark}
As shown in \cite{Zehnder75, dlL01, LlaveO99}, 
the existence of the smoothing operators implies the interpolation
  inequality, which is for any $0 \leq \lambda \leq \mu$, $0 \leq \gamma
  \leq 1$, and $v = (1 - \gamma) \lambda + \gamma \mu$, we have
  \begin{equation} \label{interpolation}
    \left\| u\right\|_v \leq C_{\gamma, \lambda, \mu} \left\| u
    \right\|_{\lambda}^{1 - \gamma} \left\| u\right\|_{\mu}^{\gamma}.
  \end{equation}

Obtaining \eqref{interpolation} as a corollary of smoothing, 
leads to the conclusion only in the case that $v$ is not an integer. 
In \cite{LlaveO99}, there is a direct proof regarding this in greater generality. 
\end{remark}

\subsection{The $\mathscr{X}^{r, \delta}$ and $\mathscr{Y}^{r, \delta}$ Space} 
\label{sec:triples} 

Our problem of solving \eqref{invariance} seeks triples of 
functions (the embedding $W$, the inner dynamics in the circle $a$ and 
the dynamics on the stable manifolds $\lambda$).  We will need spaces of triple of functions. 
In this section, we specify the topologies we have found useful. 

We now can define the scale of spaces $\mathscr{X}^{r, \delta}$ and
$\mathscr{Y}^{r, \delta}$ by the product of Banach spaces as follows:
\begin{definition}\label{def:triples} 
  Define the product space $\mathscr{X}^{r, \delta} = \mathcal{X}^{r, \delta} \times \mathcal{X}^{r, \delta}
  \times C^r \times C^r$ with norm
  $$ \left\| u \right\|_{\mathscr{X}^{r, \delta}} = \left\| W_1
  \right\|_{\mathcal{X}^{r, \delta}} + \left\| W_2 \right\|_{\mathcal{X}^{r, \delta}} +
  \left\| a \right\|_{C^r} + \left\| \lambda \right\|_{C^r}, $$
  where $u = (W_1(\theta, s), W_2(\theta, s), a(\theta), \lambda(\theta)) \in
  \mathscr{X}^{r, \delta}$. Similarly, define space $\mathscr{Y}^{r, \delta} = \mathcal{X}^{r, \delta} \times
  \mathcal{X}^{r, \delta}$ with norm
  $$ \left\| v \right\|_{\mathscr{Y}^{r, \delta}} = \left\| W_1
  \right\|_{\mathcal{X}^{r, \delta}} + \left\| W_2 \right\|_{\mathcal{X}^{r, \delta}},$$
  where $v = (W_1(\theta, s), W_2(\theta, s)) \in \mathscr{Y}^{r, \delta}$.
\end{definition}
\begin{remark}
  $\mathscr{X}^{r, \delta}, \mathscr{Y}^{r, \delta}$ are both scales of Banach spaces with smoothing
  operators. The smoothing operators comes natually from the smoothing
  operators in $C^r$ and $\mathcal{X}^r$ spaces. 
\end{remark}

\begin{remark}
  For the rest of the paper, we will always denote $u(\theta, s) \in
  \mathscr{X}^{r, \delta}$ to be the triplet $(W(\theta, s), a(\theta), \lambda(\theta))$,
  and we will not distinguish among $\left\| \cdot
  \right\|_{\mathscr{X}^{r, \delta}}$, $\left\| \cdot
  \right\|_{\mathscr{Y}^{r, \delta}}$, $\left\| \cdot \right\|_{C^r}$ and
  $\left\| \cdot \right\|_{r}$ when $\delta$ and the
  dimension of the function are understood.
\end{remark}

\section{Statement of The Analytical Result} \label{result}

In this section, we present the statement of the main result: Theorem~\ref{main}. 

As anticipated, the proof is obtained through a Nash-Moser method, 
alternating the quasi-Newton method with some smoothing steps. As discussed 
in Section~\ref{implicitremark}, the problem at hand is somewhat different 
from other previous applications of Nash-Moser technique.  The loss of 
differentiability in the estimates comes from the operator in the functional. 
The solutions of the linearized equation do not lose regularity, but they 
only work for a range of regularities. 

Since the Nash-Moser method requires alternating the quasi-Newton 
method and smoothings, we start formulating the standard setup. 
This is a scale of Banach spaces. The smoothing operators map the
spaces of less regular functions into the spaces of more regular functions
and they have quantitative properties. 

By the scale of spaces and the smoothing operators in 
Section~\ref{space}, we 
formulate our main result Theorem~\ref{main} and proceed to the proof
in Section~\ref{proof_proof}. Theorem~\ref{main} implies rather directly 
the result for foliations. We just need to verify that the operator 
entering in equation~\eqref{invariance} satisfies the hypotheses of
Theorem~\ref{main}.

As indicated in Section~\ref{implicitremark}, the implicit function 
theorem we use will require some unusual properties in Nash-Moser 
theory: We need spaces with anisotropic regularity,  the 
linearized equation does not incur any loss of regularity, 
but can only be applied in a range of regularities. This will 
require some severe adaptations from the standard expositions and the methods
based on
analytic or $C^\infty$ smoothing cannot work here.

Recall that our goal is to find $W(\theta, s), a(\theta) \text{ and } \lambda(\theta)$ 
satisfying the invariance equation~\eqref{invariance}. In other words, given $r
\geq 0, \delta > 0$, we are looking for
the zero of the functional $\mathscr{F}: \mathscr{X}^r \rightarrow \mathscr{Y}^r$
where
\begin{equation}
  \label{functional}
  \mathscr{F}[u] = \mathscr{F}[W, a, \lambda](\theta, s) = f(W(\theta, s)) -
W(a(\theta), \lambda(\theta)s), 
\end{equation}
for $u = (W, a, \lambda) \in \mathscr{X}^{r, \delta}$.

Before presenting the main Theorem~\ref{main}, we first define
\textbf{Condition-0} as follows:
\begin{definition}[Condition-0]
  \label{condition0}
  For any sufficiently small $\delta, \rho > 0$. Given $m \in \mathbb{R}$, $W: \mathbb{T}
  \times \mathbb{R} \rightarrow
  \mathbb{T} \times \mathbb{R}$, $a: \mathbb{T}  \rightarrow \mathbb{T}$ and
  $\lambda: \mathbb{T} \rightarrow \mathbb{R}$, we say that the tuple $(m, 
  W, a, \lambda)$ satisfies \textbf{Condition-0} if the following restrictions hold:
  \begin{enumerate}
  \item $\|\lambda\|_{C^0} < 1$,
  \item $(W, a, \lambda) \triangleq u \in \mathscr{X}^{m + 2, \delta}$,
  \item For $\widetilde{B}_{m+2}(\rho) \subset \mathscr{X}^{m+2}$ is the ball centered at $u
    = (W, a, \lambda)$ with radius $\rho$, 
    \begin{equation*}
      \min_{u \in \widetilde{B}_{m+2}(\rho)} \min \Big(-\frac{\ln
      \|\lambda\|_{C^0} \|(Da)^{-1}\|_{C^0}}{\ln \|Da\|_{C^0}}, -\frac{\ln
      \|\lambda\|_{C^0}}{\ln \|D(a^{-1})\|_{C^0}}, -\frac{\ln \|\lambda\|_{C^0}}{\ln
      \|Da\|_{C^0}}\Big) - 2 \geq m \geq 2.
    \end{equation*}
  \end{enumerate}
\end{definition}

\begin{remark}
  Restriction (1) can be generalized to $\lambda^* < 1$, where $\lambda^*$ is
  The Dynamical average. If $\lambda^*$ is used, one also need to adapt
  condition (3) accordingly (see Remark~\ref{general_coho}).
\end{remark}

\begin{remark}
  Restriction (3) is to guarantee $m$ is bounded above in such a way that the
  regularity requirement for solving cohomological equations \eqref{eq1_higher}, \eqref{eq2_order0}
  and \eqref{eq2_higher} covers the scale of regularities in Theorem~\ref{main}.
  (See Lemma~\ref{lemma_coho} for more details).
\end{remark}

Following the scheme derived
in Section~\ref{algorithm}, we present a theorem for the existence of solution for
$\mathscr{F}[u] = 0$:

\begin{theorem}
  \label{main}
  For sufficiently small $\delta > 0, \rho > 0$, suppose there exists a tuple $(m, W_0, a_0,
  \lambda_0)$ satisfying Condition-0.

  Let $\mathscr{X}^{r, \delta}$ and $\mathscr{Y}^{r, \delta}$ be two scales of
  Banach spaces for $m \le r \leq m + 2$.

  Consider the functional $\mathscr{F}: \widetilde{B}_r(\rho)
  \rightarrow \mathscr{Y}^r$ defined in \eqref{functional}, where
  $\widetilde{B}_r(\rho)$ is a ball centered at $u_0 \triangleq (W_0, a_0, \lambda_0) \in
  \mathscr{X}^{m+2, \delta}$ with radius $\rho$.

  If $\left\| \mathscr{F}[u_0]\right\|_{\mathscr{X}^{m - 2, \delta}}$ is
  sufficiently small, then there exists
  $u^* \in \widetilde{B}_m(\rho)$ such that $\mathscr{F}[u^*] = 0$.

  Moreover, such $u^*$ is the limit of the iteration 
  combining with some smoothing operators. The smoothing parameters 
  go to zero, and the specific rates will be given in the proof. 
  Furthermore,  the convergence of the iterations to the limit is superexponential. 

As a consequence, we have that 
\[
\| u^* - u_0 \|_{\mathscr{X}^{m, \delta}} \le C \|\mathscr{F}(u_0)
\|_{\mathscr{X}^{m-2, \delta}},
\]
where $C$ is a finite constant.
\end{theorem}

\begin{remark}
  More specifically, the restriction for $\left\|
    \mathscr{F}[u_0]\right\|_{m - 2}$ to be sufficiently small is: $$\left\|
    \mathscr{F}[u_0] \right\|_{m - 2} < e^{- 2 \mu \beta},$$ where $\mu, \beta$
  are numbers specified in the proof of Appendix~\ref{appendix}. The converging rate
  for the iteration scheme is bounded by $\left\| \mathscr{F}[u_n] \right\|_{m -
    2} \leq v e^{-2\mu\beta\kappa^n}$, with the same $\mu$ and $\beta$, and $\kappa$
  can be picked to be as close to 2 as possible.
\end{remark}

\begin{remark}
  It may seem somewhat surprising that the requirement on $\left\|
    \mathscr{F}[u_0]\right\|_{\mathscr{X}^{m - 2, \delta}}$ from lower
  regularity can result in the existence of solution $u^*$ in higher regularity
  $\mathscr{X}^{m, \delta}$, but it is actually reasonable because of the
  requirement from even higher regularity that $u_0 \in \mathscr{X}^{m + 2, \delta}$.
\end{remark}

\begin{remark}
  Since $\delta$ prescribes the range of $s$, picking a larger $\delta$
  allows us to parameterize a larger neighborhood of the invariant circle
  provided that the conditions in Theorem~\ref{main} are maintained with
  the increased $\delta$.
\end{remark}

\begin{remark} 
One of the consequences of \eqref{main} is 
that given a family of maps $f_\varepsilon$ indexed by a parameter $\varepsilon$ so 
that $f_0$ contains an invariant circle, we can design a continuation 
method by taking   the exact solution for some value of $\varepsilon$ as an 
approximate solution for $\varepsilon + \eta$ for sufficiently small $\eta$ \cite{YaoL21}. 

This procedure is guaranteed to continue till some of 
the non-degeneracy assumptions of Theorem~\ref{main} fail. 
These assumptions are just the regularity of the circle and some version of 
hyperbolicity. Hence, we know that these numerical methods will continue till 
the torus becomes irregular, the manifolds have a domain of analyticity smaller 
than $\delta$ or the hyperbolicity is lost. This may entail that the dynamical 
average gets close to $1$ (or undefined) or that the angle between the stable 
and unstable manifolds becomes zero (the bundle collapse).  Of course, 
several of the possibilities may happen at the same time.
\end{remark}

\begin{remark}
As seen in several examples (e.g. in \cite{Llave97})
one can see that the optimal regularity of the invariant circle 
may decrease continuously to 0 as the parameters change. For some 
parameter value, they will stop being $C^2$, for another parameter
they will stop being $C^1$, and then, they will become H\"older continuity.
(The isochrons remain analytic, 
even if the optimal domain may change). 

This indicates 
that the breakdown of the tori may depend on what regularity one requires to call something 
a torus. The fact that the destruction of the tori happens in a very
gradual way makes the exploration of the boundary be very subtle since 
the boundary detected depends significantly on the stopping criterion. 
For example, the destruction of the circles as $C^1$ manifolds studied 
in \cite{Mane} happens at different values  of the places 
where they disappear as $C^0$ curves or as continua
\cite{JarnikK69, CapinskiK20}. Detailed discriptions of the breakdown can be
found in \cite{YaoL21}.

Detailed numerical explorations of 
the behavior at breakdown of 
the hyperbolicity \cite{GoldenY88,Rand92a,Rand92b,HaroL06,HaroL07,CallejaF12,FiguerasH12}
has uncovered many interesting phenomena (e.g. scaling relations) 
that deserve detailed mathematical analysis. 

Of course, detailed numerical 
explorations near the boundary are very delicate and it requires having 
a very good theory (condition numbers and a-posteriori theorems) that 
ensure that the calculations are correct even when something unexpected 
is happening.
\end{remark} 

The proof of Theorem~\ref{main}  is done by verifying the
conditions of Theorem~\ref{NMIFT}.
introduced in Appendix~\ref{appendix}. Details of the proof of
Theorem~\ref{main} can be found in
section \ref{proof_proof}.

\subsection{The Analyticity Radius for $W(\theta, s)$:  a Digression}
In general, analytic radius for $W(\theta, s)$ is not infinite (for example, 
systems with more than one limit cycle) since
Algorithm~\ref{algorithm_n} computes the invariant circle and the foliations by
stable manifold in a small neighborhood of the limit cycle. On the other hand,
following \cite{Poincare90},
if the map $f$ is entire, we have the following result:
\begin{lemma}
  If the map $f: \mathbb{T} \times \mathbb{R} \rightarrow \mathbb{T} \times
  \mathbb{R}$ is an entire function, and given $(W(\theta, s), a(\theta),
  \lambda(\theta))$ $\in \mathscr{X}^{r, \delta}$ satisfies the invariance
  equation~\eqref{invariance}, we have
  that the analytic radis for $W(\theta, s)$ in $s$ is infinity.
\end{lemma}

\begin{proof}
  It follows from $W(\theta, s) \in \mathcal{X}^{r, \delta}$ that $W(\theta,
  \cdot)$ is analytic in $B_{\rho(\theta)}$, where $\rho(\theta) \geq \rho_* >
  0$. Since $f$ in entire, we also have $f \circ W(\theta, \cdot)$, thus $W(a(\theta),
  \lambda(\theta)\cdot)$ by Equation~\eqref{invariance}, is analytic
  in $B_{\rho(\theta)}$. It follows that $W(a(\theta), \cdot)$ is analytic in
  $B_{\lambda^{-1}(\theta)\rho(\theta)}$.

  By repeating the above process, one can see that
  \begin{equation*}
    W(a^{\circ m}(\theta), \cdot) \text{ is analytic in } B_{\lambda^{[m]}(\theta)^{-1}\rho(\theta)}.
  \end{equation*}

  It follows that the analyticity radius for $W(\theta, \cdot)$ is infinity.

\end{proof}

\section{Proof of the Analytical Result} \label{proof}

This section can be mainly divided into 2 parts.
In the first part, we prove two technical results. More specifically, we prove
the fibered version of the Poincar\'e-Sternberg theorem, namely the existence
of $h(\theta, s)$ in equation~\eqref{undetermination} discussed in
Section~\ref{setup_undetermincy} (see Section~\ref{proof_unique}), and we prove
the existence of the solution to the 
cohomological equation mentioned in \eqref{coho} (see Section~\ref{proof_coho}).
In the second half, we present the proof of the Theorem \ref{main}.
The idea of the proof is presented 
in Section~\ref{result}. The proof is achieved by justifying
all the non-degeneracy conditions that are listed in a modified version of the
Nash-Moser implicit function theorem (Theorem~\ref{NMIFT}), which
can be found in Appendix~\ref{appendix}. 

\subsection{The Existence of $h(\theta, s)$ in Equation
  \eqref{undetermination}} \label{proof_unique}

In this subsection, we prove Lemma~\ref{existence_h2}.
As indicated in  Section~\ref{setup_invariance}. Lemma~\ref{existence_h2}
ensures
that the study of \eqref{invariance}, which clearly is 
a sufficient condition for the existence of  foliation, 
is also necessary. This result will not be used in subsequent
studies of the existence of solutions of \eqref{invariance}.
Nevertheless, it introduces some techniques that will
be used later.  It also allows us to make some remarks
about the domains of solutions of functional equations. 

Our goal is finding $h(\theta, s)$ such that equation~\eqref{undetermination} holds for given
$\lambda(\theta, s)$ and $\widehat{\lambda}(\theta, s)$. In the following
Lemma~\ref{existence_h2},
we show the existence of $h(\theta, s)$ when $\widehat{\lambda}(\theta, s)$ equals to the
linear term of $\lambda(\theta, s)$ by the contraction mapping
theorem. 

\begin{lemma}[Existence of $h(\theta, s)$]
  \label{existence_h2}
  There exists $\delta > 0$ such that for $\widehat{\lambda} \in \mathcal{X}^{r,
    \delta}, \widehat{\lambda}(\theta, s) = \lambda(\theta)s + N(\theta, s)$, where
  $N(\theta, s) = \mathcal{O}(s^2)$.
  If there exists $k \in \mathbb{N}^+$ such that $\|\lambda^{[k]}\|_{C^0} < 1$
  and 
  $\|\lambda^{[k]}\|_{C^r} < \gamma_k$ for
  some $k \in \mathbb{N}^+$, where $\gamma_k$ is specified in the proof, then we
  have the existence of $h(\theta, s) \in \mathcal{X}^{r, \delta}$ such that
  equation~\eqref{undetermination}: $h(a(\theta),
\lambda(\theta)s) = \widehat{\lambda}(\theta, h(\theta, s))$ holds.
\end{lemma}

\begin{remark}
  The condition $\|\lambda^{[k]}\|_{C^0} < 1$ for some $k \in \mathbb{N}^+$ can
  be assured when the dynamical average $\lambda^* < 1$, and the condition
  $\|\lambda^{[k]}\|_{C^r}$ can be maintained with a suitable choice of initial
  condition $u_0$ as in Theorem~\ref{main}.
\end{remark}

\begin{remark}
  This Lemma~\ref{existence_h2} can be viewed as a \emph{``fibered''} version of the Poincar\'e-Sternberg theorem
on linearization of contractions.  We can think of $s$ as the dynamic
variable but the map sends a fiber indexed by $\theta$ into another fiber
indexed by $a(\theta)$. 

We have prepared a proof following the version of \cite{Sternberg57}
based on formulating as  contractions since it leads to concrete estimates. 
Since the maps are analytic in the dynamical variable, 
the original proof of Poincar\'e-Dulac 
\cite{Poincare78, Dulac04}  based on majorants can also be adapted. 
\end{remark}

\begin{proof}
By substituting the above $\widehat{\lambda}(\theta, s)$ and $\lambda(\theta)$ in equation
\eqref{undetermination}, we have 
\begin{equation} \label{existence_equation}
  h(a(\theta), \lambda(\theta)s) = \lambda(\theta)h(\theta, s) + N(\theta, h(\theta, s)).
\end{equation}

Since we only need the existence of $h$, we restrict ourselves for finding
$h(\theta, s)$ of the following form:
\begin{equation} \label{expression_h}
  h(\theta, s) = s + \widehat{h}(\theta, s),
\end{equation}
where $ \widehat{h}(\theta, s) = \mathcal{O}(s^2)$. 
By substituting \eqref{expression_h} back into equation~\eqref{existence_equation} and
after some simplifications, we have
\begin{equation} \label{contraction_0}
  \widehat{h}(\theta, s) = \lambda^{-1}(\theta) [\widehat{h}(a(\theta), \lambda(\theta)s) - N(\theta, s + \widehat{h}(\theta, s))].
\end{equation}

Define a Banach space
$$\widetilde{\mathcal{X}}^{r, \delta} = \Big\{ u(\theta, s) = \sum_{j =
  2}^{\infty}u^{(j)}(\theta)s^j \mid u^{(j)}(\theta) \in C^r, \text{ and }
\sum_{j = 2}^{\infty}\left\|u^{(j)}\right\|_{C^r}\delta^{j} < \infty
\Big\}.$$
We know $\widetilde{\mathcal{X}}^{r,\delta}$ is complete as it is a closed subspace of
$\mathcal{X}^{r, \delta}$, and $N(\theta, s), \widetilde{h}(\theta, s) \in
\widetilde{\mathcal{X}}^{r, \delta}$.

Denote $$\mathscr{G}[\widehat{h}] = \lambda^{-1}(\theta) [\widehat{h}(a(\theta),
\lambda(\theta)s) - N(\theta, s + \widehat{h}(\theta, s))],$$ then $\mathscr{G}:
\widetilde{\mathcal{X}}^{r, \delta} \rightarrow
\widetilde{\mathcal{X}}^{r, \delta}$. The task now is to show the existence of $\widehat{h}$ such
that $\mathscr{G}(\widehat{h}) = \widehat{h}$ through a contraction argument.

Instead of showing that $\mathscr{G}$ is a contraction, we show
$\mathscr{G}^{\circ k}$($\mathscr{G}$ compose with itself $k$ times
for some big enough integer $k$) is a contraction.

By simple calculations, one can see that
\begin{equation} \label{G_m}
    \mathscr{G}^{\circ k}[\widehat{h}] = (\lambda^{-1})^{[k]}(\theta) [\widehat{h}(a^{\circ k}(\theta),
    \lambda^{[k]}(\theta)s) - k \mathcal{O}{(s^2)}],
  \end{equation}
where the second term $k \mathscr{O}(s^2)$ is formed by the summation of $n$
terms of $N(\cdot,
\cdot)$, each is of $\mathcal{O}{(s^2)}$, which can be controlled to be small by
some upper bound $\delta_o$ since $|s| < \delta < \delta_0$.

It remains to show that the first term of \eqref{G_m}:
$(\lambda^{-1})^{[k]}(\theta) \widehat{h}(a(\theta), \lambda^{[k]}(\theta)s)
\triangleq \mathscr{L}[\widehat{h}]$ is a contraction. 
For every $\hat{h}_1, \hat{h}_2 \in \widetilde{\mathcal{X}}^r$, we have
$\hat{h}_1(\theta, s) = \sum_{j = 2}^{\infty}\hat{h}_1^{(j)}(\theta)s^j$ and
$\hat{h}_2(\theta, s) = \sum_{j = 2}^{\infty}\hat{h}_2^{(j)}(\theta)s^j$. We have
\begin{align*}
  \left\|\mathscr{L}[\widehat{h}_1] - \mathscr{L}[\widehat{h}_2]\right\|_{\mathcal{X}^{r, \delta}} &= \left\|(\lambda^{-1})^{[k]}(\theta)(\hat{h}_1(a^{\circ k}(\theta), \lambda^{[k]}(\theta)s) - \hat{h}_2(a^{\circ k}(\theta), \lambda^{[k]}(\theta)s))) \right\|_{\mathcal{X}^{r, \delta}}\\
                                                                           &\leq \left\| \sum_{j = 2}^{\infty}(\hat{h}_1^{(j)} - \hat{h}_2^{(j)})(a^{\circ k}(\theta)) \lambda^{[k - 1](j - 1)}(\theta) s^j \right\|_{\mathcal{X}^{r, \delta}} \\
                                                                                                   &\leq C_{r, k, \|a\|_{C^r}, \|\lambda\|_{C^0}}  \left\| \lambda^{[k]} \right\|_{C^r}^r \sum_{j = 2}^{\infty} \left\| \hat{h}_1^{(j)} - \hat{h}_2^{(j)} \right\|_{\mathcal{X}^{r, \delta}}  \delta^{j}\\
                                                                                                   &\leq \zeta \left\| \hat{h}_1 - \hat{h}_2\right\|_{\mathcal{X}^{r, \delta}}
 \end{align*}

provided that $\|\lambda^{[k]}\|_{C^r} < (\zeta C_{r, k, \|a\|_{C^r}, \|\lambda\|_{C^0}}^{-1})^{\frac{1}{r}} \triangleq \gamma_k$ for any $0 < \zeta <
1$,  where the second last
inequality is achieved by utilizing Lemma~\ref{inequal_more_Cr} and $C_{r, k,
  \|a\|_{C^r}, \|\lambda\|_{C^0}}^{-1} > 0$  is a constant related to $r, k$,
$\|a\|_{C^r}$  and $\|\lambda\|_{C^0}$ only.

By the above discussion, we have the existence of a unique $\hat{h}^* \in \tilde{X}^r$
such that $\mathscr{G}(\hat{h}^*) = \hat{h}^*$, which finishes the proof.

\end{proof}

\subsection{Estimates on Solutions of the Cohomological Equation~\eqref{coho}} \label{proof_coho}

We use this subsection to take a closer look at the cohomological equation
mentioned in \eqref{coho} with solution \eqref{coho_solution}. The following result in
Lemma~\ref{lemma_coho} is used in both 
Section~\ref{algorithm_derivation} and Section~\ref{proof}. 
\begin{lemma}
  \label{lemma_coho}
  Given $l(\theta), a(\theta)$ and
  $\eta(\theta) \in C^r$ with $\|l\|_{C^0} < 1$. If $r < -
  \ln\|l\|_{C^0} / \ln\|Da\|_{C^0}$ (i.e. $\left\|Da
  \right\|_{C^0}^r\left\| l\right\|_{C^0} < 1$), then the cohomological equation
  \eqref{coho}: $\phi(\theta) = l(\theta)\phi(a(\theta)) + \eta(\theta)$
  admits a unique $C^r$ solution:
  \begin{equation} \label{coho_sol}
    \phi(\theta) = \sum_{j = 0}^{\infty}l^{[j]}(\theta)\eta(a^j(\theta))
  \end{equation}
  with
  $$ \left\| \phi \right\|_{C^r} \leq C_{l, a, r} \left\| \eta\right\|_{C^r}
  \leq \infty,$$
\end{lemma}

\begin{proof}
  First, we prove that \eqref{coho_sol} is a solution to equation~\eqref{coho}.
  Since $\|l\|_{C^0} < 1$ and $\|\eta\|_{C^0}$ is bounded,
  by noticing that $\sum_{j = 0}^{\infty}l^{[j]}(\theta)\eta(a^j(\theta))$
  converges uniformly in $C^0$, 
  one can substitute this infinite sum back in \eqref{coho} and rearrange terms to show
  that \eqref{coho_sol} is indeed a solution.

  On top of this, we argue that \eqref{coho_sol} is the only $C^0$ solution.
  More explictly, if there were two solutions, then by the discussion in
  \eqref{coho_steps}: $\phi(\theta) = \sum_{j =
    0}^{n}l^{[j]}(\theta)\eta(a^j(\theta)) + l^{[n + 1]}(\theta)\phi(a^{n +
    1}(\theta))$, they would agree on the first $n$ terms, and since the limit
  of the $C^0$ norm for the last term goes to $0$ as $n$ goes to infinity, the
  two solutions are the same.

  To finish the proof, we now show that $\|\phi\|_{C^r} < \infty$. 
  By Lemma~\ref{inequal_Halpha} and Lemma~\ref{inequal_more_Cr}, we have the
  following inequalities,
  \begin{enumerate}
  \item $\|a^{\circ k}\|_{C^r} \leq k^n n! \|Da\|_{C^0}^{r(k
      -1)}\|a\|_{C^r}^{r+1},$
  \item $\|\eta(a^{\circ k})\|_{C^r} \leq n!
    k^{n-1}(n+nk+1)\|\eta\|_{C^r}\|a\|_{C^r}^{r+1}\|Da\|_{C^0}^{kr},$
  \item $\|l^{[k]}\|_{C^r} \leq k^{n+1}(n+1)!(\|l\|_{C^r} +
    \|a\|_{C^r})^{r+1}\|l\|_{C^0}^{max(0, k-n-1)}\|Da\|_{C^0}^{kr},$
  \item $\|\eta (a^{\circ k}) l^{[k]}\|_{C^r} \leq C_r(\|l\|_{C^r} +
    \|a\|_{C^r})^{r+1} \|l\|_{C^0}^{-n}\big[k^n(\|l\|_{C^0}\|Da\|_{C^0}^{r})^k\big] \|\eta \|_{C^r}$
  \end{enumerate}
  Thus from \eqref{coho_solution}, we have
  \begin{equation*}
    \sum_{j = 0}^{\infty} \left\|
      l^{[j]}\eta(a^{\circ j}) \right\|_{C^r} 
    \leq C_r (\left\|l \right\|_{C^r} + \left\|
      a\right\|_{C^r})^{r+1}\|l\|_{C^r}^{-n} (\sum_{j=1}^{\infty}j^{n}(\left\|Da \right\|_{C^0}^r\left\| l\right\|_{C^0})^j) \left\| \eta \right\|_{C^r},
  \end{equation*}
  thus if $r < - \ln \|l\|_{C^0} / \ln \|Da\|_{C^0}$, we have $\left\|l
    \right\|_{C^0}\left\|Da \right\|_{C^0}^{r} < 1$. 
    It follows that
    $$\|\phi\|_{C^r} \leq \sum_{j=0}^{\infty} \|l^{[j]}\eta(a^{\circ
      j})\|_{C^r} < \infty, $$
  which finishes the proof.
\end{proof}

\begin{remark} \label{general_coho}
  Give $k \in \mathbb{N}^+$, by rewriting equation~\eqref{coho} into the form as
  in \eqref{coho_steps}, i.e.,
  $$ \phi(\theta) = l^{[k + 1]}(\theta)\phi(a^{\circ (k+1)}(\theta)) + \sum_{j = 0}^{k}l^{[j]}(\theta)\eta(a^{\circ j}(\theta)),$$
  The requirement for $r$ can be generalized slightly to be $r < - \ln \|l^{[k]}\|_{C^0}
  / \ln \|D(a^{\circ (k+1)})\|_{C^0}$, we have $\left\|l^{[k]}
    \right\|_{C^0}\left\|D(a^{\circ (k+1)}) \right\|_{C^0}^{r} < 1$,
\end{remark}
\begin{remark}
  Lemma~\ref{lemma_coho} shows that if
  \begin{equation}
    \label{finite_reg_2}
    \left\|l \right\|_{C^0} \left\|Da \right\|_{C^0}^{r} < 1 \text{ (or } \left\|l^{[k]}
    \right\|_{C^0}\left\|D(a^{\circ (k+1)}) \right\|_{C^0}^{r} < 1,)
  \end{equation}
  then we have that $\phi(\theta) = \sum_{j =
    0}^{\infty}l^{[j]}(\theta)\eta(a^j(\theta))$ converges absolutely in the
  $C^r$ sense, thus $\phi \in C^r$.

  Note that the condition \eqref{finite_reg_2} can only be satisfied for a finite range
  of regularity $r$. We now give examples to show that this
  condition is sharp. 

If $a(\theta)$ has an attractive fixed point, which we place
at $\theta = 0$.
If $a(\theta) = \lambda \theta$ in a neigborhood and, moreover $l(\theta)$ is 
a constant, we see that \eqref{coho_solution} becomes
\begin{equation}\label{weierstrass} 
  \phi(\theta) = \sum_{j = 0}^{\infty}l^j\eta(\lambda^j \theta))
\end{equation}

which is a version of  the classical Weierstrass function, which for 
even polynomial $\eta$ can be arranged to be  finite differentiable, 
showing that the range claimed in Lemma~\ref{lemma_coho} is optimal
in the generality claimed. 
Indeed the map that in local coordinates has the expression 
$(x,y)  = \lambda x, l y + \eta(x) $ has an invariant circle 
given by the graph of the function $\phi$ in \eqref{weierstrass}. 

The fact that one can only solve the cohomology equations for
a certain range of regularities makes it impossible to use 
the Nash-Moser methods that are based on approximating 
solutions of $C^\infty$ or $C^\omega$ problems. 
\end{remark}

\subsection{Proof for Theorem~\ref{main}}
\label{proof_proof}

Following the same notation as in Theorem \ref{main}, we now justify the
non-degeneracy conditions of the abstract Nash-Moser Theorem \ref{NMIFT} one by
one.

\begin{lemma}[Condition 1]
  \label{lemma_condition1}
  For $\delta, \widetilde{B}_r(\rho)$ defined in Theorem~\ref{main}, we have
  $\mathscr{F} (\widetilde{B}_r(\rho) \cap \mathscr{X}^{r, \delta}) \subset
  \mathscr{Y}^{r, \delta}$. 
\end{lemma}
\begin{proof}
  For every $u(\theta, s) = (W(\theta, s), a(\theta), \lambda(\theta)) \in
  \widetilde{B}_r(\rho) \cap \mathscr{X}^{r, \delta}$, recall
  $\widetilde{B}_r(\rho)$ is a ball with radius
  $\rho$, we have
  $$ \mathscr{F}[u](\theta, s) = f \circ W(\theta, s) - W(a(\theta),
  \lambda(\theta)s).$$
  
  First, we show that $f \circ W(\theta, s) \in \mathscr{Y}^r$. 
  With no loss of generality, we will only consider the
  first component of $f = (f_1, f_2)$ and show that
  $\left\| f_1(W_1, W_2) \right\|_{\mathscr{X}^{r, \delta}} < \infty$.

  Write $$f_1(\theta, s) = \sum_{j = 0}^{\infty} f_1^{(j)}(\theta) s^j, ~
  W_1(\theta, s) = \sum_{j = 0}^{\infty} W_1^{(j)}(\theta) s^j, ~W_2(\theta, s) =
  \sum_{j = 0}^{\infty} W_2^{(j)}(\theta) s^j.$$
  Notice that
  \begin{align*}
    f_1^{(j)}(W_1(\theta, s)) &= f_1^{(j)}(W_1^{(0)}(\theta)) +
  \Big(\frac{d}{d \theta}f_1^{(j)}\Big)(W_1^{(0)}(\theta))\Big(\sum_{j = 1}^{\infty}W_1^{(j)}(\theta)s^j\Big) + \ldots \\ &+
  \frac{1}{k!}\Big(\frac{d^k}{d \theta}f_1^{(j)}\Big)(W_1^{(0)}(\theta))\Big(\sum_{j = 1}^{\infty}W_1^{(j)}(\theta)s^j\Big)^k + \ldots
  \end{align*}
  since $f^{(j)}(\theta)$ is analytic, we can treat it as a function in
  $\mathbb{C}$, and then by Cauchy's estimates for derivatives, we have
  $$
  \left| \frac{d^k}{d \theta}(f_1^{(j)})(W_1^{(0)}(\theta)) \right| \leq
  \frac{k!}{R^k} \max_{z \in \gamma_R} \left| f_1^{(j)}(z) \right| =
  \frac{k!}{R^k}C_R, \forall R > 0.
  $$
  where $\gamma_R = \{z \ | \ | z - W_1^{(0)}(\theta) | = R\}$.
  It follows that
  \begin{align} 
    \left\| f^{(j)}(W_1) \right\|_{\mathscr{X}^{r, \delta}} &\leq C_R(1 + R^{-1}\left\| W_1\right\|_{\mathscr{X}^{r, \delta}} + R^{-2}\left\| W_1\right\|_{\mathscr{X}^{r, \delta}}^2 + \ldots) \nonumber\\
    &\leq C_R\left(\frac{1}{1 - \frac{\left\| W_1\right\|_{\mathscr{X}^{r, \delta}}}{R}}\right) \leq  C_R \frac{1}{1 - \frac{\rho}{R}} \label{fbound}
  \end{align}
  Thus
  \begin{align*}
    \left\| f_1(W_1, W_2) \right\|_{\mathscr{X}^{r, \delta}} &= \left\| \sum_{j = 0}^{\infty} f_1^{(j)}(W_1) (W_2 s)^j\right\|_{\mathscr{X}^{r, \delta}} \\
                                     &\leq \sum_{j=0}^{\infty}\left\| f^{(j)}(W_1)\right\|_{\mathscr{X}^{r, \delta}} \left\| W_2 s\right\|_{\mathscr{X}^{r, \delta}}^j (2^r)^j \\
                                     &\leq C_R\left(\frac{1}{1 - \frac{\rho}{R}}\right) \sum_{j=0}^{\infty}(2^r\rho\delta)^j\\
                                     &< \infty,
  \end{align*}

  where the third line is because of \eqref{fbound} and $$\left\| W_2 s\right\|_{\mathscr{X}^{r, \delta}} = \left\| \sum_{j = 0}^{\infty}
    W_2^{(j)} s^{j+1}\right\|_{\mathscr{X}^{r, \delta}} = \sum_{j=0}^{\infty} \left\|
    W_2^{(j)}\right\|_{\mathscr{X}^{r, \delta}} \delta ^{(j + 1)} = \left\| W_2\right\|_{\mathscr{X}^{r, \delta}} \delta,$$
  and the last line is because of the assumption on $\rho$ in Theorem \ref{main}.

  It remains to show that $\left\| W(a, \lambda s)\right\|_r < \infty$, this is
  trivial since
  \begin{align*}
    \left\| W(a, \lambda s) \right\|_{\mathscr{X}^{r, \delta}} &= \left\| \sum_{j=0}^{\infty}W^{(j)}(a)\lambda^js^j\right\|_{\mathscr{X}^{r, \delta}} \\
                                                               &\leq \sum_{j = 0}^{\infty} \left\| W^{(j)}(a)\lambda^j \right\|_{C^r} \delta^{j} \leq \sum_{j=0}^{\infty}(2^r)^j \left\| W^{(j)}(a)\right\|_{C^r} \left\| \lambda\right\|_{C^r}^j \delta^{j} \\
                                                               & \leq 2 M_r \|a\|_{C^r}^r 2^{2n+1} \max_{0\leq k\leq r}(\|\lambda\|_{C^r}^k) \max_{0 \leq j < \infty}(\|\lambda\|_{C^0}^{\max(j -n - 1, 0)} j^{2(n - 1)}) \|W\|_{\mathcal{X}^{r, \delta}} \\
    & < \infty.
  \end{align*}
  for $(W(\theta, s), a(\theta), \lambda(\theta)) \in \widetilde{B}_r(\rho)$ and
  $\|\lambda\|_{C^0} < 1.$
\end{proof}

\begin{lemma}[Condition 2]
  \label{lemma_condition2}
  $\mathscr{F}|_{\tilde{B}_m\cap\mathscr{X}^r}: \widetilde{B}_r(\rho) \cap
  \mathscr{X}^r \rightarrow \mathscr{X}^{r}$ has continuous first and second order
  Fr\'echet derivatives, and satisfy the following conditions:
  \begin{itemize}
  \item[$\ast$] $\left\| D\mathscr{F}[u](h) \right\|_{m-2} \leq C_{r, \widetilde{B}_r(\rho)} \left\|
      h\right\|_{m-2}$ for $h \in \mathscr{X}^m$.
  \item[$\ast$] $\left\| D^2\mathscr{F}[u](h)(k)\right\|_{m-2} \leq C_{r, \widetilde{B}_r(\rho)} \left\|
      h\right\|_{m-1}\left\| k\right\|_{m-1}$ for $k, h \in \mathscr{X}^m.$
  \end{itemize}
  where $C_{r, \widetilde{B}_r(\rho)}$ is a constant depends on the regularity
  and the ball $\widetilde{B}_r(\rho) \in \mathscr{X}^{r, \delta}$ only.
\end{lemma}
\begin{proof}
  By some routine calculation, for $h = (h_1, h_2, h_3)$, $k = (k_1, k_2, k_3) \in
  \mathscr{X}^r$, where $h_1, k_1 \in \mathcal{X}^r \times \mathcal{X}^r$, $h_2, h_3, k_2, k_3 \in
  C^r$, we can calculate the first and second order
  Fr\'echet derivatives as follows:
  \begin{align*}
    D\mathscr{F}[u](h) &= Df(W)h_1 - \partial_1W(a, \lambda)h_2 - \partial_2W(a, \lambda)h_3 - h_1(a, \lambda),\\
    D^2\mathscr{F}[u](k, h) &= D^2f(W)(k_1, h_1) - \partial_{11}W(a, \lambda)(k_2, h_2) - \partial_{12}W(a, \lambda)(k_3, h_2) \\
                     &- \partial_1k_1(a, \lambda) h_2 - \partial_{21}W(a, \lambda)(k_2, h_3) - \partial_{22}W(a, \lambda)(k_3, h_3) \\
                     &- \partial_2k_1(a, \lambda) h_3 - \partial_1h_1(a, \lambda) k_2 - \partial_2h_1(a, \lambda) k_3
  \end{align*}
  Thus we have
  \begin{align*}
    \left\| D\mathscr{F}[u](h) \right\|_{m-2} \leq& 2^{2 n + 1} (\left\| Df(W)\right\|_{m-2} \left\| h_1\right\|_{m-2} + \left\| \partial_1W(a, \lambda)\right\|_{m-2} \left\| h_2\right\|_{m-2})\\
                                               &+ \left\| \partial_2 W(a, \lambda)\right\|_{m-2} \left\| h_3\right\|_{m-2}) + \left\| h_1(a, \lambda)\right\|_{m-2}\\
    \leq& C_{r, \widetilde{B}_r(\rho)} \left\| h \right\|_{m-2}.
  \end{align*}
  and
  \begin{align*}
    \left\| D^2\mathscr{F}[u](k, h) \right\|_{m-2} \leq& \left\|D^2f(W)(k_1, h_1)\right\|_{m-2} - \left\|\partial_{11}W(a, \lambda)(k_2, h_2)\right\|_{m-2}\\ 
    &- \left\|\partial_{12}W(a, \lambda)(k_3, h_2)\right\|_{m-2} 
                                                    - \left\|\partial_1k_1(a, \lambda) h_2\right\|_{m-2} \\
                                                    &- \left\|\partial_{21}W(a, \lambda)(k_2, h_3)\right\|_{m-2} - \left\|\partial_{22}W(a, \lambda)(k_3, h_3)\right\|_{m-2} \\
                                                    &- \left\|\partial_2k_1(a, \lambda) h_3\right\|_{m-2} - \left\|\partial_1h_1(a, \lambda) k_2\right\|_{m-2} - \left\|\partial_2h_1(a, \lambda) k_3\right\|_{m-2}.\\
    \leq& ~ C_{r, \widetilde{B}_r(\rho)}\big(\left\| h \right\|_{m-2} \left\| k \right\|_{m-2}  + \left\| h\right\|_{m-1} \left\| k \right\|_{m-2} +
          \left\| h \right\|_{m-2} \left\| k \right\|_{m-1}  \\
    &+ \left\| h\right\|_{m-1} \left\| k \right\|_{m-1}\big)\\
    \leq& C_{r, \widetilde{B}_r(\rho)} \left\| h \right\|_{m-1} \left\| k \right\|_{m-1}
  \end{align*}
\end{proof}

\begin{lemma}[Condition 3]
  \label{lemma_condition3}
  For $u \in \widetilde{B}_r(\rho)$ and $r = m - 2$, $m + 2$, we have 
\[ \left\| \eta[u] \mathscr{F}[u] \right\|_{r} \leq C_{r, \widetilde{B}_r(\rho)} \left\|  \mathscr{F}[u]\right\|_r.
\]
where $\eta[u]$ serves as the approximate inverse of the derivative of the
  functional $\mathscr{F}[u]$, which is defined in our algorithm in Section~\ref{algorithm}. 
\end{lemma}
\begin{proof}
  Note that we only need to apply $\eta[u]$ on the range of 
$\mathscr{F}[u]$ and we do not need estimates on the whole space. 
In contrast with other Nash-Moser implicit function theorems, 
the operator $\eta[u]$ is bounded from spaces to themselves and does 
not entail any loss of regularity. 
  
  From Lemma~\ref{lemma_coho}, by equation~\eqref{eq1_higher}, \eqref{eq2_order0}
  and \eqref{eq2_higher}, we have
  $$\left\| \Gamma_1 \right\|_r \leq C_{r, \widetilde{B}_r(\rho)} \left\| \widetilde{e}_1\right\|_r $$
  and
  $$ \left\| \Gamma_2 \right\|_r \leq C_{r, \widetilde{B}_r(\rho)} (\left\|
    \widetilde{e}_1\right\|_r + \left\| \widetilde{e}_2\right\|_r), $$
  from equation~\eqref{eq1_order0} and \eqref{eq2_order1}, we also have
  $$ \left\| \Delta_a\right\|_r \leq C_{r, \widetilde{B}_r(\rho)} \left\| \widetilde{e}_1
  \right\|_r$$
  and
  $$ \left\| \Delta_{\lambda} \right\|_r \leq C_{r, \widetilde{B}_r(\rho)} \left\|
    \widetilde{e}_2 \right\|_r.$$
  Together with $\left\| \widetilde{e} \right\|_r \leq C_{r, \widetilde{B}_r(\rho)} \left\| e \right\|_r$, which can
  be shown trivialy,  we have $\left\| \Delta_W\right\|_r \leq C_{r, \widetilde{B}_r(\rho)}
  \left\| e
  \right\|_r$, $\left\| \Delta_a\right\|_r \leq C_{r, \widetilde{B}_r(\rho)} \left\| e \right\|_r$, and
  $\left\| \Delta_{\lambda} \right\|_r \leq C_{r, \widetilde{B}_r(\rho)} \left\| e
  \right\|_r$, which finishes the proof.
\end{proof}

\begin{lemma}[Condition 4]
  \label{lemma_condition4}
  For $u \in \tilde{B}_m$, we have $$\left\| (D\mathscr{F}[u]\eta[u] - Id)
    \mathscr{F}[u] \right\|_{m - 2} \leq C_{r, \widetilde{B}_r(\rho)} \left\| \mathscr{F}[u] \right\|_m \left\|  \mathscr{F}[u]\right\|_{m -1}.$$ 
\end{lemma}
\begin{proof}
  \begin{align*}
    \mathscr{F}[u]& - D\mathscr{F}[u]\eta[u] \mathscr{F}[u]\\ &= \mathscr{F}[u] + D\mathscr{F}[u]\Delta_u\\  &= \mathscr{F}[u + \Delta_u] + \mathscr{O}(\Delta ^ 2)\\
    &= f(W + \Delta_W) - (W + \Delta_W)(a + \Delta_a, (\lambda + \Delta_{\lambda})s) + \mathscr{O}(\Delta ^ 2)\\
    &= f(W) + Df(W)\Delta_W - W(a, \lambda s) - DW(a, \lambda s)\begin{pmatrix}
        \Delta_a \\
        \Delta_{\lambda}s
      \end{pmatrix} - \Delta_W(a, \lambda s) \\
     & \quad + D\Delta_W(a, \lambda s)\begin{pmatrix}
        \Delta_a \\
        \Delta_{\lambda} s
      \end{pmatrix} + \mathscr{O}(\Delta ^ 2)\\
    &= -DW(a, \lambda s)\Big[ \begin{pmatrix}
      Da & 0 \\ D\lambda s & \lambda
    \end{pmatrix} \Gamma - \begin{pmatrix} \Delta_a \\ \Delta_{\lambda} s 
    \end{pmatrix} - \Gamma(a, \lambda s) - \widetilde{e}
    \Big]  \\ & \quad - De \Gamma  
      + D\Delta_W(a, \lambda s)\begin{pmatrix}
        \Delta_a \\
        \Delta_{\lambda} s
      \end{pmatrix} + \mathscr{O}(\Delta ^ 2)\\
    &= - De \Gamma  + D\Delta_W(a, \lambda s)\begin{pmatrix}
        \Delta_a \\
        \Delta_{\lambda} s
      \end{pmatrix} + \mathscr{O}(\Delta ^ 2).
  \end{align*}
  By the proof in Lemma~\ref{lemma_condition3}, and that $\left\|
    De\right\|_{m-2} \leq C_{r, \widetilde{B}_r(\rho)} \left\| e \right\|_{m-1}$, $\left\|
    D\Delta_W\right\|_{m-1} \leq C_{r, \widetilde{B}_r(\rho)} \left\| e \right\|_{m-1}$, 
  we achieve
  \begin{equation*}
    \left\| (D\mathscr{F}[u]\eta[u] - Id) \mathscr{F}[u] \right\|_{m - 2} \leq C_{r, \widetilde{B}_r(\rho)}\left\| \mathscr{F}(u) \right\|_{m - 1} \left\| \mathscr{F}(u)\right\|_m. 
  \end{equation*}
\end{proof}

\begin{lemma}[Condition 5]
  \label{lemma_condition5}
  For $u \in \tilde{B}_m \cap \mathscr{X}^{m + 2}$, we have $\left\| \mathscr{F}[u] \right\|_{m + 2} \leq C_{r, \widetilde{B}_r(\rho)}(1 + \left\|
    u\right\|_{m + 2})$.
\end{lemma}
\begin{proof}
  From the proof of Lemma~\ref{lemma_condition1} above, we can see that there
  exists a constant $C > 0$ such that $\left\| \mathscr{F}[u] \right\|_r < C_{r, \widetilde{B}_r(\rho)}$
  for $u \in \tilde{B}_m \cap \mathscr{X}^{m+2}$, thus the Lemma follows naturally.
\end{proof}

Since all the constant $C_{r, \widetilde{B}_r(\rho)}$ we get from Lemma~\ref{lemma_condition1},
\ref{lemma_condition2}, \ref{lemma_condition3}, \ref{lemma_condition4} and
\ref{lemma_condition5} are universal for $u$ in the respective domain $\widetilde{B}_r(\rho)$, we have
finished proving all the non-degeneracy conditions
required by Theorem \ref{NMIFT} in the Appendix \ref{appendix}. Thus we
have proved Theorem \ref{main}.

\appendix
\section{An Abstract Implicit Function Theorem in Scales of Banach Spaces}
\label{appendix}

In this appendix, we present and prove Theorem~\ref{main}, which is a modified
version of the Nash-Moser implicit function theorem. We have made the assumptions
in Theorem~\ref{main} 
to match the inequalities that we can achieve from the algorithm
in Section~\ref{algorithm}. We hope that this theorem can also be applied in
some other problems involving invariance equations in the theory of 
normally hyperbolic systems. 

The main idea of the Nash-Moser smoothing technique is to add a smoothing
operation inside the Newton steps. That is, even though the Newton (or
quasi-Newton) steps lose regularities, the smoothing operator restores them.

As anticipated in Section~\ref{implicitremark}, our 
problem has some unusual properties which make it impossible 
to use other results. 
As peculiarities of the analysis our problem we 
recall:
\begin{enumerate} 
\item 
The functional we are trying to solve is not differentiable from one space
to itself (It is basically, the composition operator). 
\item
The linearized equation can be solved without loss of regularity, but 
only for regularities on a range. This range does not change much by 
smoothing the problem. Hence, the technique of approximating the problem 
by $C^\infty$ or analytic ones does not produce any results.
A result we found inspiring is \cite{Schwartz60}. 
\item
The use of identities to simplify the equation leads to an extra term in the error 
estimates after applying the iterative method. 
The new error contains a term  estimated by a derivative 
of the original error multiplied by the correction (in appropriate 
norms).  Implicit function theorems 
with these terms were already considered in \cite{Vano02, CallejaL10, CallejaCL13} 
but they were treated by analytic or $C^\infty$ smoothing which 
is not possible for the problem in this paper. 
\item
In the problem at hand it is natural to use functions with a mixed regularity: finitely 
differentiable in one variable and analytic in another. 
\end{enumerate}

The statement of the abstract Nash-Moser implicit function theorem we will 
use is: 

\begin{theorem}
  \label{NMIFT}
  Let $m > 2$ and $\mathscr{X}^r$, $\mathscr{Y}^r$ for $m \leq r \leq m + 2$ be
  scales of Banach spaces with smoothing operators. Let $B_r$ be the unit ball
  in $\mathscr{X}^r$, $\widetilde{B}_r(\rho) = u_0 + \rho B_r$ be the unit ball translated by
  $u_0 \in \mathscr{X}^r$ with radius scaled by $\rho > 0$, and $B(\mathscr{Y}^r, \mathscr{X}^{r})$ is
  the space of bounded linear operators from $\mathscr{Y}^r$ to
  $\mathscr{X}^{r}$. Consider a map
  $$ \mathscr{F}: \widetilde{B}_r(\rho) \rightarrow \mathscr{Y}^r$$
  and
  $$ \eta: \tilde{B}_{r} \rightarrow B(\mathscr{Y}^{r}, \mathscr{X}^{r})$$
  satisfing:
  \begin{itemize}
  \item $\mathscr{F}(\widetilde{B}_r(\rho) \cap \mathscr{X}^r) \subset
    \mathscr{Y}^{r}$ for $m \leq r \leq m + 2$.
  \item $\mathscr{F}|_{\tilde{B}_m\cap\mathscr{X}^r}: \widetilde{B}_r(\rho) \cap
    \mathscr{X}^r \rightarrow \mathscr{X}^{r}$ has two continuous
    Fr\'echet derivatives, and satisfy the following bounded conditions:
    \begin{itemize}
    \item[$\ast$] $\left\| D\mathscr{F}[u](h) \right\|_{m-2} \leq C \left\|
        h\right\|_{m-2}$ for $h \in \mathscr{X}^m$.
    \item[$\ast$] $\left\| D^2\mathscr{F}[u](h)(k)\right\|_{m-2} \leq C \left\|
        h\right\|_{m-1}\left\| k\right\|_{m-1}$ for $k, h \in \mathscr{X}^m.$
    \end{itemize}
  \item $\left\| \eta[u] \mathscr{F}[u] \right\|_{r} \leq C \left\|  \mathscr{F}[u]\right\|_r$,
    $u \in \widetilde{B}_r(\rho)$ for $r = m - 2, m + 2$.
  \item $\left\| (D\mathscr{F}[u]\eta[u] - Id)  \mathscr{F}[u] \right\|_{m -
      2} \leq
    \left\| \mathscr{F}[u] \right\|_m \left\|  \mathscr{F}[u]\right\|_{m -1}$,
    $u \in \tilde{B}_m$.
    
  \item $\left\| \mathscr{F}[u] \right\|_{m + 2} \leq C(1 + \left\|
      u\right\|_{m + 2}), u \in \tilde{B}_m \cap \mathscr{X}^{m + 2}$.
  \end{itemize}
  Then if $\left\| \mathscr{F}[u_0]\right\|_{m-2}$ is sufficiently
  small, then there exists $u^* \in \mathscr{X}^m$ such that
  $\mathscr{F}[u^*] = 0$. Moreover, 
  $$ \left\| u_0 - u^*\right\|_m \leq C \left\|
    \mathscr{F}[u_0]\right\|_{m-2}$$
\end{theorem}

\begin{proof}
  Let $\kappa > 1, \beta, \mu, \alpha > 0, 0 < v < 1$ be real numbers to be
  specified later. Consider the sequence $u_n$ such that
  \begin{equation}
    \label{smooth_newton}
    u_n = u_{n-1} - S_{t_{n-1}} \eta[u_{n-1}] \mathscr{F}[u_{n-1}],
  \end{equation}
  where $t_n = e^{\beta \kappa^{n-1}}$.
  We will prove that this sequence satisfies the following three conditions
  inductively:
  \begin{enumerate}[start = 1, label={($P_{\arabic*}$n):}]
  \item $u_n \in \tilde{B}_m$,
  \item $\left\| \mathscr{F}[u_n] \right\|_{m - 2} \leq v e^{-2\mu\beta\kappa^n}$,
  \item $1 + \left\| u_n \right\|_{m + 2} \leq v e^{2\alpha\beta\kappa^n}$.
  \end{enumerate}

  First, for $n = 0$, we know $P_1(n=0)$ is ture automatically. By setting $v =
  \left\| \mathscr{F}[u_0] \right\|_{m - 2} e^{2 \mu \beta}$ with $\mu, \beta$
  be specified later and $\left\| \mathscr{F}[u_0] \right\|_{m - 2} < e^{- 2 \mu
    \beta}$,
  $P_2(n=0)$ is true. Given $\alpha$, we can let $\beta$ be big
  enough such that condition $P_3(n=0)$: $1 + \left\| u_0 \right\|_{m + 2} \leq
  e^{2\alpha\beta}$ holds.
  Now, suppose $P_1, P_2$ and $P_3$ are true for $n - 1$, we will now show that
  the three conditions are true for $n$.

  By assumption (3) and $P_2(n-1)$, we have
  \begin{equation} \label{etaf}
    \left\| \eta[u_n] \mathscr{F}[u_n] \right\|_{\mathscr{X}^{m-2}} \leq C \left\| \mathscr{F}[u_n] \right\|_{\mathscr{Y}^{m-2}} \leq Cve^{-2\mu\beta\kappa^n},
  \end{equation}
  it follows from \eqref{smooth_newton} and \eqref{etaf} that
  \begin{align}
    \label{unun}
    \left\| u_n - u_{n-1} \right\|_{\mathscr{X}^{m}} &= \left\| S_{t_{n-1}} \eta[u_{n-1}] \mathscr{F}[u_{n-1}] \right\|_{\mathscr{X}^{m}}\nonumber \\
                                                     &\leq C t_{n-1}^2 \left\| \eta[u_{n-1}] \mathscr{F}[u_{n-1}]\right\|_{\mathscr{X}^{m-2}} \nonumber \\
                                                     &\leq C v e^{2\beta\kappa^{n-1}(1 - \mu)}.
  \end{align}
  where the second inequality above comes from \eqref{smoothing1}.

  Thus we have
  \begin{align} \label{NMcondition1}
    \left\| u_n - u_0 \right\|_{\mathscr{X}^{m}} &\leq \sum_{j=1}^{\infty} \left\| u_j - u_{j-1} \right\|_{\mathscr{X}^{m}} \nonumber \\
                                                 &\leq Cv(e^{2\beta(1-\mu)} + e^{2\beta(1-\mu)\kappa} + e^{2\beta(1-\mu)\kappa^2} + e^{2\beta(1-\mu)\kappa^3} + e^{2\beta(1-\mu)\kappa^4} + \sum_{j=6}^{\infty}e^{2\beta(1-\mu)\kappa^j})\nonumber \\
                                                 &\leq Cv(e^{2\beta(1-\mu)} + e^{2\beta(1-\mu)\kappa} + e^{2\beta(1-\mu)\kappa^2} + e^{2\beta(1-\mu)\kappa^3} + e^{2\beta(1-\mu)\kappa^4} + \frac{e^{2\beta(1-\mu)\kappa^6}}{1 - e^{2\beta(1-\mu)\kappa}})\nonumber \\
                                                 &\leq \rho.
  \end{align}
  where the third inequality comes from the fact that $\kappa^{j-1} > j\kappa$
  for $j \geq 5$ and $\kappa > \sqrt[3]{5}$, and the last inequality can be
  achieved if $\mu > 1$ and $\beta$ is large enough. Thus we have proved $P_1(n)$.

  In order to prove $P_2(n)$, let us break $\mathscr{F}[u_n]$ as follows:
  \begin{align} \label{NMcondition2_0}
    \left\| \mathscr{F}[u_n]\right\|_{\mathscr{Y}^{m-2}} \leq& \left\| \mathscr{F}[u_n] - \mathscr{F}[u_{n-1}] + D\mathscr{F}[u_n] S_{t_{n-1}} \eta[u_{n-1}] \mathscr{F}[u_{n-1}]\right\|_{\mathscr{Y}^{m-2}} \nonumber \\
                                              &+ \left\|(Id - D\mathscr{F}[u_{n-1}]\eta[u_{n-1}]) \mathscr{F}[u_{n-1}] \right\|_{\mathscr{Y}^{m-2}} \nonumber \\
                                              &+ \left\| D\mathscr{F}[u_{n-1}](Id - S_{t_{n-1}}) \eta[u_{n-1}] \mathscr{F}[u_{n-1}] \right\|_{\mathscr{Y}^{m-2}}
  \end{align}
  and estimates the three terms one by one:
  
  For the first line, by Taylor expansion, the induction condition in the
  second part of assumption (2), \eqref{smoothing1} and \eqref{etaf}, we have 
  \begin{align} \label{NMcondition2_1}
    l_1 &= \left\| \mathscr{F}[u_n] - \mathscr{F}[u_{n-1}] + D\mathscr{F}[u_n] S_{t_{n-1}} \eta[u_{n-1}] \mathscr{F}[u_{n-1}]\right\|_{\mathscr{Y}^{m-2}} \nonumber \\
    &\leq C \left\|D^2\mathscr{F}[u_n](S_{t_{n-1}}\eta[u_{n-1}]\mathscr{F}[u_{n-1}])(S_{t_{n-1}}\eta[u_{n-1}]\mathscr{F}[u_{n-1}]) \right\|_{\mathscr{Y}^{m-2}} \nonumber \\
    &\leq C \left\| (S_{t_{n-1}}\eta[u_{n-1}]\mathscr{F}[u_{n-1}])\right\|_{\mathscr{Y}^{m-1}}^2 \nonumber \\ 
    &\leq C t_{n-1}^2 \left\| \eta[u_{n-1}\mathscr{F}[u_{n-1}]]\right\|_{\mathscr{Y}^{m-2}}^2 \nonumber \\
    &\leq Cve^{(1 - 2\mu)2\beta\kappa^{n-1}}
  \end{align}

  For the second line, by assumption(4), \eqref{interpolation}, $P_2(n - 1)$,
  assumption (5) and $P_3(n-1)$,
  we have
  \begin{align} \label{NMcondition2_2}
    l_2 &= \left\| (Id - D\mathscr{F}[u_{n-1}]\eta[u_{n-1}]) \mathscr{F}[u_{n-1}] \right\|_{\mathscr{Y}^{m-2}} \nonumber \\
        &\leq C \left\| \mathscr{F}[u_{n-1}] \right\|_{\mathscr{Y}^{m-1}} \left\| \mathscr{F}[u_n]\right\|_{\mathscr{Y}^{m}} \nonumber \\
        &\leq C \left\| \mathscr{F}[u_{n-1}]\right\|_{\mathscr{Y}^{m-2}}^{\frac{5}{4}} \left\| \mathscr{F}[u_{n-1}] \right\|_{\mathscr{Y}^{m+2}}^{\frac{3}{4}} \nonumber \\
        &\leq C \left\| \mathscr{F}[u_{n-1}]\right\|_{\mathscr{Y}^{m-2}}^{\frac{5}{4}} (1 + \left\| u_{n-1} \right\|_{\mathscr{Y}^{m+2}})^{\frac{3}{4}} \nonumber \\
        &\leq C v^{\frac{3}{2}} e^{(-\frac{5}{2}\mu + \frac{3}{2}\alpha) \beta \kappa^{n-1}}
  \end{align}

  For the third line, by \eqref{smoothing2}, \eqref{etaf} and assumption (5), we
  have
  \begin{align} \label{NMcondition2_3}
    l_3 &= \left\| D\mathscr{F}[u_{n-1}](Id - S_{t_{n-1}}) \eta[u_{n-1}] \mathscr{F}[u_{n-1}] \right\|_{\mathscr{Y}^{m-2}} \nonumber \\
        &\leq C t_{n-1}^4 \left\| \eta[u_{n-1}] \mathscr{F}[u_{n-1}]\right\|_{\mathscr{Y}^{m+2}} \nonumber \\
        &\leq C t_{n-1}^4 (1 + \left\| u_{n-1}\right\|_{\mathscr{Y}^{m+2}}) \nonumber \\
        &\leq C v e^{2\beta\kappa^{n-1}(\alpha - 2)}
  \end{align}

  thus, in order to show that \eqref{NMcondition2} is true, we want $l_1 + l_2 +
  l_3 \leq v e^{-2\mu\beta\kappa^n}$, i.e.
  \begin{equation*}
    (Cve^{(1 - 2\mu)2\beta\kappa^{n-1}} + C v^{\frac{3}{2}} e^{(-\frac{5}{2}\mu + \frac{3}{2}\alpha) \beta \kappa^{n-1}} + v e^{2\beta\kappa^{n-1}(\alpha - 2)}) < v e^{-2\mu\beta\kappa^n},
  \end{equation*}
  Thus we need
  \begin{equation} \label{NMcondition2}
    C(e^{(1-2\mu+\mu \kappa) 2\beta\kappa^{n-1}} + v^{\frac{1}{2}}e^{(-3\mu + \alpha + \mu \kappa) 2\beta \kappa^{n-1}} + v e^{(\alpha -2 + \kappa)2\mu\beta\kappa^{n-1}}) < 1,
  \end{equation}
  which can be satisfied if $\mu, \kappa$ and $\alpha$ satisfies
  \begin{equation}
    \begin{cases}
      1 - 2\mu + \mu \kappa < 0, \\
      -3\mu + \alpha + \mu \kappa < 0,\\
      \alpha - 2 + \kappa < 0.
    \end{cases}
  \end{equation}
  and $\beta$ is picked large enough.

  As for $P_3(n)$, by \eqref{unun}, \eqref{smooth_newton},
  \eqref{smooth_newton}, \eqref{etaf}, assumption (5) and $P_3(j)$ for $j < n -
  1$, we have
  \begin{align} 
    1 + \left\| u_n\right\|_{\mathscr{X}^{m+2}} &\leq 1 + \sum_{j=1}^n \left\| u_j - u_{j-1}\right\|_{m+2} \nonumber \\
                                                &\leq 1 + \sum_{j=1}^n \left\| S_{t_{j-1}}\eta[u_{j-1}] \mathscr{F}[u_{j-1}] \right\|_{m+2} \nonumber \\
                                                &\leq 1 + C \sum_{j=1}^n \left\| \eta[u_{j-1} \mathscr{F}[u_{j-1}]]\right\|_{m+2} \nonumber \\
                                                &\leq 1 + C \sum_{j=1}^n (1 + \left\|u_{j-1}\right\|_{\mathscr{X}^{m+2}}) \nonumber \\
    &\leq 1 + C \sum_{j=1}^n e^{2\alpha\beta\kappa^{j-1}}
  \end{align}
  we need
  \begin{equation*}
    (1 + \left\| u_n\right\|_{m+2}) e^{-2\alpha\beta\kappa^n} < 1,
  \end{equation*}
  that is
  \begin{equation} \label{NMcondition3}
    e^{-2\alpha\beta\kappa^n} + C \sum_{j = 1}^n e^{(1-\kappa)\alpha\beta\kappa^{j-1}} < v.
  \end{equation}
  which is
  \begin{equation*}
    e^{-2\alpha\beta\kappa^n} + C \sum_{j = 1}^n e^{(1-\kappa)\alpha\beta\kappa^{j-1}} < \left\| \mathscr{F}[u_0] \right\|_{m - 2}.
  \end{equation*}
  
  By the same reason as in \eqref{NMcondition1}, we have
  $$
  \sum_{j=1}^n e^{(1-\kappa)\alpha\beta\kappa^{j-1}} \leq
  e^{(1-\kappa)\alpha\beta} + e^{(1-\kappa)\alpha\beta\kappa} +
  e^{(1-\kappa)\alpha\beta\kappa^2} + e^{(1-\kappa)\alpha\beta\kappa^3} +
  e^{(1-\kappa)\alpha\beta\kappa^4} + \frac{e^{(1-\kappa)\alpha\beta6\kappa}}{1
    - e^{(1-\kappa)\alpha\beta\kappa}}
  $$
  can be achieved if $\kappa > \sqrt[3]{5}$ and $\beta$ is large enough.

  Above all, in oder to make sure that \eqref{NMcondition1},
  \eqref{NMcondition2} and \eqref{NMcondition3} are true, we nee the following
  constrictions for $\kappa, \alpha$ and $\mu$:
  \begin{equation} \label{constrictions}
    \begin{cases}
      \mu > 1, \\
      \kappa > \sqrt[3]{5}, \\
      1 - 2\mu + \mu \kappa < 0, \\
      -3\mu + \alpha + \mu \kappa < 0, \\
      \alpha - 2 + \kappa < 0.
    \end{cases}
  \end{equation}
  and $\beta$ is large enough.

  One possible solution for \eqref{constrictions} is $\kappa = 1.75, \mu = 5$
  and $\alpha = 0.05$ and $\beta$ is large enough.

  Up to this point, we have finished the proof for induction. By letting $n
  \rightarrow \infty$, the second assumption
  $\left\|\mathscr{F}[u_n]\right\|_{\mathscr{X}^{m-2}} \leq
  ve^{-2\mu\beta\kappa^n}$ leads to a solution $u^* \in \mathscr{X}^{m-2}$ such that
  $\mathscr{F}[u^*] = 0$, and the convergence is superexponential. Moreover, by
  the discussion in \eqref{NMcondition1}, we have
  $$
  \left\| u^* - u_0\right\|_m \leq Cv = C \left\|
    \mathscr{F}[u_0]\right\|_{m-2},
  $$
  which completes the proof.
\end{proof}
\begin{remark}
  Although the result $\left\| u^* -
    u_0\right\|_m \leq Cv = C \left\| \mathscr{F}[u_0]\right\|_{m-2}$ is a bit
  surprising in the sense that the
  higher regularity norm is bounded by the lower one, but this inequality is
  actually justified by the bounds from even higher regularity required in
  the assumption.
\end{remark}

%% \bibliographystyle{alpha}\newcommand{\etalchar}[1]{$^{#1}$}

%% \bibliography{theorical_ref}{}

\begin{thebibliography}{CFdlL03b}

\bibitem[Ado07]{Adomaitis07}
Raymond~A. Adomaitis.
\newblock The trouble with spurious eigenvalues.
\newblock {\em Internat. J. Bifur. Chaos Appl. Sci. Engrg.}, 17(4):1375--1381,
  2007.

\bibitem[AO00]{AinsworthO}
Mark Ainsworth and J.~Tinsley Oden.
\newblock {\em A posteriori error estimation in finite element analysis}.
\newblock Pure and Applied Mathematics (New York). Wiley-Interscience [John
  Wiley \& Sons], New York, 2000.

\bibitem[AVK87]{AndronovVK87}
A.~A. Andronov, A.~A. Vitt, and S.~\`E. Kha\u{\i}kin.
\newblock {\em Theory of oscillators}.
\newblock Dover Publications, Inc., New York, 1987.
\newblock Translated from the Russian by F. Immirzi, Reprint of the 1966
  translation.

\bibitem[BDV05]{BonattiDV05}
Christian Bonatti, Lorenzo~J. D\'{\i}az, and Marcelo Viana.
\newblock {\em Dynamics beyond uniform hyperbolicity}, volume 102 of {\em
  Encyclopaedia of Mathematical Sciences}.
\newblock Springer-Verlag, Berlin, 2005.
\newblock A global geometric and probabilistic perspective, Mathematical
  Physics, III.

\bibitem[BLZ00]{BatesLZ00}
Peter~W. Bates, Kening Lu, and Chongchun Zeng.
\newblock Invariant foliations near normally hyperbolic invariant manifolds for
  semiflows.
\newblock {\em Trans. Amer. Math. Soc.}, 352(10):4641--4676, 2000.

\bibitem[BLZ08]{BatesLZ08}
Peter~W. Bates, Kening Lu, and Chongchun Zeng.
\newblock Approximately invariant manifolds and global dynamics of spike
  states.
\newblock {\em Invent. Math.}, 174(2):355--433, 2008.

\bibitem[BS08]{BjerklovS08}
Kristian Bjerkl\"{o}v and Maria Saprykina.
\newblock Universal asymptotics in hyperbolicity breakdown.
\newblock {\em Nonlinearity}, 21(3):557--586, 2008.

\bibitem[BST98]{BroerST98}
Henk Broer, Carles Sim\'{o}, and Joan~Carles Tatjer.
\newblock Towards global models near homoclinic tangencies of dissipative
  diffeomorphisms.
\newblock {\em Nonlinearity}, 11(3):667--770, 1998.

\bibitem[CCdlL13]{CallejaCL13}
Renato~C. Calleja, Alessandra Celletti, and Rafael de~la Llave.
\newblock A {KAM} theory for conformally symplectic systems: efficient
  algorithms and their validation.
\newblock {\em J. Differential Equations}, 255(5):978--1049, 2013.

\bibitem[CdlL10]{CallejaL10}
Renato Calleja and Rafael de~la Llave.
\newblock A numerically accessible criterion for the breakdown of
  quasi-periodic solutions and its rigorous justification.
\newblock {\em Nonlinearity}, 23(9):2029--2058, 2010.

\bibitem[CF12]{CallejaF12}
Renato Calleja and Jordi-Llu\'{\i}s Figueras.
\newblock Collision of invariant bundles of quasi-periodic attractors in the
  dissipative standard map.
\newblock {\em Chaos}, 22(3):033114, 10, 2012.

\bibitem[CFdlL03a]{CFdlL03a}
Xavier Cabr\'{e}, Ernest Fontich, and Rafael de~la Llave.
\newblock The parameterization method for invariant manifolds. {I}. {M}anifolds
  associated to non-resonant subspaces.
\newblock {\em Indiana Univ. Math. J.}, 52(2):283--328, 2003.

\bibitem[CFdlL03b]{CFdlL03b}
Xavier Cabr\'{e}, Ernest Fontich, and Rafael de~la Llave.
\newblock The parameterization method for invariant manifolds. {II}.
  {R}egularity with respect to parameters.
\newblock {\em Indiana Univ. Math. J.}, 52(2):329--360, 2003.

\bibitem[CFdlL05]{CFdlL05}
Xavier Cabr\'{e}, Ernest Fontich, and Rafael de~la Llave.
\newblock The parameterization method for invariant manifolds. {III}.
  {O}verview and applications.
\newblock {\em J. Differential Equations}, 218(2):444--515, 2005.

\bibitem[CH17a]{CanadellH17a}
Marta Canadell and \`Alex Haro.
\newblock Computation of quasi-periodic normally hyperbolic invariant tori:
  algorithms, numerical explorations and mechanisms of breakdown.
\newblock {\em J. Nonlinear Sci.}, 27(6):1829--1868, 2017.

\bibitem[CH17b]{CanadellH17b}
Marta Canadell and \`Alex Haro.
\newblock Computation of quasiperiodic normally hyperbolic invariant tori:
  rigorous results.
\newblock {\em J. Nonlinear Sci.}, 27(6):1869--1904, 2017.

\bibitem[CJ15]{ChungJ15}
Y.-M. Chung and M.~S. Jolly.
\newblock A unified approach to compute foliations, inertial manifolds, and
  tracking solutions.
\newblock {\em Math. Comp.}, 84(294):1729--1751, 2015.

\bibitem[CK20]{CapinskiK20}
Maciej~J. Capi\'{n}ski and Hieronim Kubica.
\newblock Persistence of normally hyperbolic invariant manifolds in the absence
  of rate conditions.
\newblock {\em Nonlinearity}, 33(9):4967--5005, 2020.

\bibitem[CZ15]{CapinskiZ15}
Maciej~J. Capi\'{n}ski and Piotr Zgliczy\'{n}ski.
\newblock Geometric proof for normally hyperbolic invariant manifolds.
\newblock {\em J. Differential Equations}, 259(11):6215--6286, 2015.

\bibitem[dlL97]{Llave97}
Rafael de~la Llave.
\newblock Invariant manifolds associated to nonresonant spectral subspaces.
\newblock {\em J. Statist. Phys.}, 87(1-2):211--249, 1997.

\bibitem[dlL01]{dlL01}
Rafael de~la Llave.
\newblock A tutorial on {KAM} theory.
\newblock In {\em Smooth ergodic theory and its applications ({S}eattle, {WA},
  1999)}, volume~69 of {\em Proc. Sympos. Pure Math.}, pages 175--292. Amer.
  Math. Soc., Providence, RI, 2001.

\bibitem[dlLK19]{LlaveK19}
Rafael de~la Llave and Florian Kogelbauer.
\newblock Global persistence of lyapunov subcenter manifolds as spectral
  submanifolds under dissipative perturbations.
\newblock {\em {SIAM} J. Appl. Dyn. Syst.}, 18(4):2099--2142, 2019.

\bibitem[dlLO99]{LlaveO99}
R.~de~la Llave and R.~Obaya.
\newblock Regularity of the composition operator in spaces of {H}\"{o}lder
  functions.
\newblock {\em Discrete Contin. Dynam. Systems}, 5(1):157--184, 1999.

\bibitem[dlLW11]{dlLW11}
Rafael de~la Llave and Alistair Windsor.
\newblock Smooth dependence on parameters of solutions to cohomology equations
  over {A}nosov systems with applications to cohomology equations on
  diffeomorphism groups.
\newblock {\em Discrete Contin. Dyn. Syst.}, 29(3):1141--1154, 2011.

\bibitem[Dul03]{Dulac04}
Henri Dulac.
\newblock {\em Recherches sur les points singuliers des \'equations
  differentielles}.
\newblock Gauthier-Villars, 1903.

\bibitem[ET10]{ErmentroutT10}
G.~Bard Ermentrout and David~H. Terman.
\newblock {\em Mathematical foundations of neuroscience}, volume~35 of {\em
  Interdisciplinary Applied Mathematics}.
\newblock Springer, New York, 2010.

\bibitem[Fen77]{Fenichel77}
Neil Fenichel.
\newblock Asymptotic stability with rate conditions. {II}.
\newblock {\em Indiana Univ. Math. J.}, 26(1):81--93, 1977.

\bibitem[Fen72]{Fenichel71}
Neil Fenichel.
\newblock Persistence and smoothness of invariant manifolds for flows.
\newblock {\em Indiana Univ. Math. J.}, 21:193--226, 1971/72.

\bibitem[Fen74]{Fenichel73}
Neil Fenichel.
\newblock Asymptotic stability with rate conditions.
\newblock {\em Indiana Univ. Math. J.}, 23:1109--1137, 1973/74.

\bibitem[FH12]{FiguerasH12}
Jordi-Llu\'{\i}s Figueras and \`Alex Haro.
\newblock Reliable computation of robust response tori on the verge of
  breakdown.
\newblock {\em SIAM J. Appl. Dyn. Syst.}, 11(2):597--628, 2012.

\bibitem[GE88]{GoldenY88}
Melinda~P. Golden and B.~{Erik Ydstie}.
\newblock Bifurcation in model reference adaptive control systems.
\newblock {\em Systems \& Control Letters}, 11(5):413--430, 1988.

\bibitem[Gra17]{Granados}
Albert Granados.
\newblock Invariant manifolds and the parameterization method in coupled energy
  harvesting piezoelectric oscillators.
\newblock {\em Phys. D}, 351/352:14--29, 2017.

\bibitem[HCF{\etalchar{+}}16]{H16}
\`Alex Haro, Marta Canadell, Jordi-Llu\'{\i}s Figueras, Alejandro Luque, and
  Josep-Maria Mondelo.
\newblock {\em The parameterization method for invariant manifolds}, volume 195
  of {\em Applied Mathematical Sciences}.
\newblock Springer, [Cham], 2016.
\newblock From rigorous results to effective computations.

\bibitem[HdlL06a]{HaroL06}
\`A. Haro and R.~de~la Llave.
\newblock Manifolds on the verge of a hyperbolicity breakdown.
\newblock {\em Chaos}, 16(1):013120, 8, 2006.

\bibitem[HdlL06b]{HdlL06b}
\`A. Haro and R.~de~la Llave.
\newblock A parameterization method for the computation of invariant tori and
  their whiskers in quasi-periodic maps: numerical algorithms.
\newblock {\em Discrete Contin. Dyn. Syst. Ser. B}, 6(6):1261--1300, 2006.

\bibitem[HdlL07]{HaroL07}
A.~Haro and R.~de~la Llave.
\newblock A parameterization method for the computation of invariant tori and
  their whiskers in quasi-periodic maps: explorations and mechanisms for the
  breakdown of hyperbolicity.
\newblock {\em SIAM J. Appl. Dyn. Syst.}, 6(1):142--207, 2007.

\bibitem[HdlL13]{HdlL13}
Gemma Huguet and Rafael de~la Llave.
\newblock Computation of limit cycles and their isochrons: fast algorithms and
  their convergence.
\newblock {\em SIAM J. Appl. Dyn. Syst.}, 12(4):1763--1802, 2013.

\bibitem[Izh07]{Izhikevich07}
Eugene~M. Izhikevich.
\newblock {\em Dynamical systems in neuroscience: the geometry of excitability
  and bursting}.
\newblock Computational Neuroscience. MIT Press, Cambridge, MA, 2007.

\bibitem[JK69]{JarnikK69}
Ji\v{r}\'{\i} Jarn\'{\i}k and Jaroslav Kurzweil.
\newblock On invariant sets and invariant manifolds of differential systems.
\newblock {\em J. Differential Equations}, 6:247--263, 1969.

\bibitem[Lev81]{Levi81}
Mark Levi.
\newblock Qualitative analysis of the periodically forced relaxation
  oscillations.
\newblock {\em Mem. Amer. Math. Soc.}, 32(244):vi+147, 1981.

\bibitem[Mat68]{Mather68}
John~N. Mather.
\newblock Characterization of {A}nosov diffeomorphisms.
\newblock {\em Nederl. Akad. Wetensch. Proc. Ser. A 71 = Indag. Math.},
  30:479--483, 1968.

\bibitem[Min62]{Minorsky62}
Nicolas Minorsky.
\newblock {\em Nonlinear oscillations}.
\newblock D. Van Nostrand Co., Inc., Princeton, N.J.-Toronto-London-New York,
  1962.

\bibitem[MM76]{MarsdenM76}
J.~E. Marsden and M.~McCracken.
\newblock {\em The {H}opf bifurcation and its applications}.
\newblock Springer-Verlag, New York, 1976.
\newblock With contributions by P. Chernoff, G. Childs, S. Chow, J. R. Dorroh,
  J. Guckenheimer, L. Howard, N. Kopell, O. Lanford, J. Mallet-Paret, G. Oster,
  O. Ruiz, S. Schecter, D. Schmidt and S. Smale, Applied Mathematical Sciences,
  Vol. 19.

\bibitem[Mn78]{Mane}
Ricardo Ma\~{n}\'{e}.
\newblock Persistent manifolds are normally hyperbolic.
\newblock {\em Trans. Amer. Math. Soc.}, 246:261--283, 1978.

\bibitem[Mos66a]{Moser66b}
J\"{u}rgen Moser.
\newblock A rapidly convergent iteration method and non-linear differential
  equations. {II}.
\newblock {\em Ann. Scuola Norm. Sup. Pisa Cl. Sci. (3)}, 20:499--535, 1966.

\bibitem[Mos66b]{Moser66a}
J\"{u}rgen Moser.
\newblock A rapidly convergent iteration method and non-linear partial
  differential equations. {I}.
\newblock {\em Ann. Scuola Norm. Sup. Pisa Cl. Sci. (3)}, 20:265--315, 1966.

\bibitem[Nei59]{N59}
Ju~Neimark.
\newblock On some cases of periodic motions depending on parameters.
\newblock {\em Dokl. Akad. Nauk SSSR}, 129(4):736--739, 1959.

\bibitem[Pes04]{Pesin04}
Yakov~B. Pesin.
\newblock {\em Lectures on partial hyperbolicity and stable ergodicity}.
\newblock Zurich Lectures in Advanced Mathematics. European Mathematical
  Society (EMS), Z\"{u}rich, 2004.

\bibitem[Poi79]{Poincare78}
Henri Poincar\'e.
\newblock {\em Sur les propri\'et\'es des fonctions d\'efinies par les
  \'equations aux differences partielles}.
\newblock Gauthier Villars, 1879.

\bibitem[Poi90]{Poincare90}
H.~Poincar\'e.
\newblock Sur une classe nouvelle de transcendentes uniformes.
\newblock {\em Jour. de Math.}, 6:313--365, 1890.

\bibitem[Ran92a]{Rand92a}
D.~A. Rand.
\newblock Existence, nonexistence and universal breakdown of dissipative golden
  invariant tori. {I}. {G}olden critical circle maps.
\newblock {\em Nonlinearity}, 5(3):639--662, 1992.

\bibitem[Ran92b]{Rand92b}
D.~A. Rand.
\newblock Existence, nonexistence and universal breakdown of dissipative golden
  invariant tori. {II}. {C}onvergence of renormalization for mappings of the
  annulus.
\newblock {\em Nonlinearity}, 5(3):663--680, 1992.

\bibitem[Ran92c]{Rand92c}
D.~A. Rand.
\newblock Existence, nonexistence and universal breakdown of dissipative golden
  invariant tori. {III}. {I}nvariant circles for mappings of the annulus.
\newblock {\em Nonlinearity}, 5(3):681--706, 1992.

\bibitem[RT71]{RuelleT71}
David Ruelle and Floris Takens.
\newblock On the nature of turbulence.
\newblock {\em Comm. Math. Phys.}, 20:167--192, 1971.

\bibitem[Sac64]{Sacker64}
Robert~John Sacker.
\newblock {\em O{N} {INVARIANT} {SURFACES} {AND} {BIFURCATION} {OF} {PERIODIC}
  {SOLUTIONS} {OF} {ORDINARY} {DIFFERENTIAL} {EQUATIONS}}.
\newblock ProQuest LLC, Ann Arbor, MI, 1964.
\newblock Thesis (Ph.D.)--New York University.

\bibitem[Sac09]{S64}
Robert~J. Sacker.
\newblock On invariant surfaces and bifurcation of periodic solutions of
  ordinary differential equations. {C}hapter {II}. {B}ifurcation-mapping
  method.
\newblock {\em J. Difference Equ. Appl.}, 15(8-9):759--774, 2009.
\newblock Reprinted from New York Univ. Report IMM-NYU 333, October 1964,
  Courant Inst., New York.

\bibitem[Sch60]{Schwartz60}
J.~Schwartz.
\newblock On {N}ash's implicit functional theorem.
\newblock {\em Comm. Pure Appl. Math.}, 13:509--530, 1960.

\bibitem[Ste57]{Sternberg57}
Shlomo Sternberg.
\newblock Local contractions and a theorem of {P}oincar\'{e}.
\newblock {\em Amer. J. Math.}, 79:809--824, 1957.

\bibitem[Ste70]{Stein70}
Elias~M. Stein.
\newblock {\em Singular integrals and differentiability properties of
  functions}.
\newblock Princeton Mathematical Series, No. 30. Princeton University Press,
  Princeton, N.J., 1970.

\bibitem[Sza20]{Szalai19}
Robert Szalai.
\newblock Invariant spectral foliations with applications to model order
  reduction and synthesis.
\newblock {\em Nonlinear Dynamics}, 101(4):2645--2669, 2020.

\bibitem[Van02]{Vano02}
John~Andrew Vano.
\newblock {\em A {N}ash-{M}oser implicit function theorem with {W}hitney
  regularity and applications}.
\newblock ProQuest LLC, Ann Arbor, MI, 2002.
\newblock Thesis (Ph.D.)--The University of Texas at Austin.

\bibitem[Win01]{Winfree01}
Arthur~T. Winfree.
\newblock {\em The geometry of biological time}, volume~12 of {\em
  Interdisciplinary Applied Mathematics}.
\newblock Springer-Verlag, New York, second edition, 2001.

\bibitem[Win75]{W74}
A.~T. Winfree.
\newblock Patterns of phase compromise in biological cycles.
\newblock {\em J. Math. Biol.}, 1(1):73--95, 1974/75.

\bibitem[WY02]{WangY02}
Qiudong Wang and Lai-Sang Young.
\newblock From invariant curves to strange attractors.
\newblock {\em Comm. Math. Phys.}, 225(2):275--304, 2002.

\bibitem[WY03]{WangY03}
Qiudong Wang and Lai-Sang Young.
\newblock Strange attractors in periodically-kicked limit cycles and {H}opf
  bifurcations.
\newblock {\em Comm. Math. Phys.}, 240(3):509--529, 2003.

\bibitem[YdlL21]{YaoL21}
Yian Yao and Rafael de~la Llave.
\newblock Computing the invariant circle and the foliation by stable manifolds
  for 2-d maps by the parameterization method: Numerical implementation and
  results.
\newblock 2021.
\newblock Manuscript in progress.

\bibitem[ZdlL18]{ZdlL18}
Lei Zhang and Rafael de~la Llave.
\newblock Transition state theory with quasi-periodic forcing.
\newblock {\em Commun. Nonlinear Sci. Numer. Simul.}, 62:229--243, 2018.

\bibitem[Zeh75]{Zehnder75}
E.~Zehnder.
\newblock Generalized implicit function theorems with applications to some
  small divisor problems. {I}.
\newblock {\em Comm. Pure Appl. Math.}, 28:91--140, 1975.

\end{thebibliography}

\end{document}